\documentclass[12pt]{amsart}

\usepackage{enumerate,url,amssymb,  nicefrac, mathrsfs, graphicx, pdfsync}
\newtheorem{theorem}{Theorem}[section]
\newtheorem{lemma}[theorem]{Lemma}
\newtheorem*{lemma*}{Lemma}

\newtheorem{proposition}[theorem]{Proposition}

\theoremstyle{definition}
\newtheorem{definition}[theorem]{Definition}
\newtheorem{example}[theorem]{Example}

\newtheorem{question}[theorem]{Question}

\theoremstyle{remark}
\newtheorem{remark}[theorem]{Remark}

\numberwithin{equation}{section}

\newcommand{\abs}[1]{\lvert#1\rvert}
\newcommand{\Abs}[1]{\Big\lvert#1\Big\rvert}
\newcommand{\norm}[1]{\lVert#1\rVert}

\newcommand{\C}{\mathbb{C}}

\newcommand{\W}{\mathscr{W}}

\newcommand{\E}{\mathcal{E}}

\newcommand{\R}{\mathbb{R}}
\newcommand{\X}{\mathbb{X}}
\newcommand{\U}{\mathbb{U}}

\newcommand{\Y}{\mathbb{Y}}

\newcommand{\dtext}{\textnormal d}

\newcommand{\onto}{\xrightarrow[]{{}_{\!\!\textnormal{onto\,\,}\!\!}}}
\newcommand{\into}{\xrightarrow[]{{}_{\!\!\textnormal{into\,\,}\!\!}}}
\newcommand{\bydef}{\stackrel {\textnormal{def}}{=\!\!=} }

\DeclareMathOperator{\re}{Re}
\DeclareMathOperator{\im}{Im}

\DeclareMathOperator{\loc}{loc}

\DeclareMathOperator{\diff}{Diff}

\def\le{\leqslant}
\def\ge{\geqslant}

\begin{document}

\title[Hopf-Harmonics]{Monotone Hopf-Harmonics}

\author[T. Iwaniec]{ Tadeusz Iwaniec}
\address{Department of Mathematics, Syracuse University, Syracuse,
NY 13244, USA}
\email{tiwaniec@syr.edu}

\author[J. Onninen]{Jani Onninen}
\address{Department of Mathematics, Syracuse University, Syracuse,
NY 13244, USA}
\email{jkonnine@syr.edu}
\thanks{ T. Iwaniec was supported by the NSF grant DMS-1802107.
J. Onninen was supported by the NSF grant DMS-1700274.}




\maketitle

\begin{abstract}

The present paper introduces the concept of monotone Hopf-harmonics in 2D as an alternative to harmonic homeomorphisms. It opens a new area of study in Geometric Function Theory (GFT).  Much of the foregoing is motivated by the principle of non-interpenetration of matter in the mathematical theory of Nonlinear Elasticity (NE). The question we are concerned with is whether or not a Dirichlet energy-minimal mapping between Jordan domains with a prescribed boundary homeomorphism remains injective in the domain. The classical theorem of Rad\'{o}-Kneser-Choquet asserts that this is the case when the target domain is convex. An alternative way to deal with arbitrary target domains is to minimize the Dirichlet energy subject to only homeomorphisms and their limits. This leads to the so called \textit{Hopf-Laplace equation}. Among its solutions (some rather surreal) are continuous monotone mappings of Sobolev class  $\mathscr W^{1,2}_{\textnormal{loc}}$,  called \textit{monotone Hopf-harmonics}. It is at the heart of the present paper to show that such solutions are correct generalizations of harmonic homeomorphisms and, in particular, are legitimate deformations of hyperelastic materials  in the modern theory of NE. We make this clear by means of several examples.
\end{abstract}


\section{Introduction}

Throughout this text $\,\mathbb X\,$ and $\,\mathbb Y\,$ are bounded simply connected Jordan domains in the complex plane $\mathbb C$. Their boundaries $\,\partial\mathbb X\,$ and $\,\partial\mathbb Y\,$ are  positively oriented (counterclockwise) simple closed curves; when traveling in such direction the domains remain in the left hand side. We are concerned with orientation preserving homeomorphisms $\,h : \overline{\mathbb X} \onto \overline{\mathbb Y}\,$ of Sobolev class $\,\mathscr W^{1,2}(\mathbb X, \mathbb R^2)\,$ and their uniform limits.
The greatest lower bound of the Dirichlet energy is applicable to all such homeomorphisms:
\[
\mathscr E_\X[h]  \bydef \int_\mathbb X |Dh(x)|^2 \,\textnormal{d} x\; \geqslant  \,2 \int_\mathbb X \textnormal{det} Dh(x) \,\textnormal{d} x\;=  2\, |\mathbb Y|
\]

Equality occurs if and only if $\,h : \mathbb X \onto \mathbb Y\,$ is conformal whose  existence is guaranteed by
the Riemann mapping theorem.  Every conformal map  $\,f : \mathbb X \onto \mathbb Y\,$ between Jordan domains extends as a homeomorphism between the closed regions, still denoted by $\,f : \overline{\mathbb X} \onto \overline{\mathbb Y}\,$.
In other words, conformal mappings solve the so-called \textit{frictionless} minimization problem~\cite{Bac, Ba11, Cib, CN}. This means that the mappings in question are allowed to slide along the boundary (no constraints on the boundary values).
However, prescribing arbitrarily the  boundary data of a conformal mapping is an ill-posed problem. This pertains not only to the Cauchy-Riemann equations but also  to all first order elliptic systems in the complex plane. The situation is dramatically different if we move to the realm of second order PDEs, such as  complex-valued harmonic mappings $\,h = u + i \,v\,$ in which $\,u\,$ and $\,v\,$ need not be harmonic conjugates.  There always exists a unique harmonic extension of a continuous boundary map. When the target domain $\, \mathbb Y\,$ is convex the celebrated theorem of Rad\'{o}-Kneser-Choquet ~\cite{DurenBook}  asserts  that the extension is a homeomorphism.

\begin{theorem}\label{RKC} \textnormal{(RKC-Theorem)}
Let $\Y$ be a convex domain in $\mathbb C$ and $g \colon \partial \X \onto \partial \Y$  a homeomorphism. Then there exists a unique harmonic homeomorphism $h \colon \X \onto \Y$  (actually $\mathscr C^\infty$-diffeomorphism) which extends continuously up to  $\,\partial \mathbb X\,$  and coincides with $g$ on $\partial \X$.
\end{theorem}

In contrast to the case of harmonic conjugates it is not true  that a  harmonic extension of a homeomorphism  $\, h :\partial \mathbb X  \onto \partial \mathbb Y\,$
 gives rise to  a homeomorphism $\, h : \mathbb X \onto \mathbb Y\,$.  Even more precise statement holds, if the target $\Y$ is not convex there always exists a boundary homeomorphism  $\,h : \partial \mathbb X \onto \partial  \mathbb Y\,$
whose harmonic extension takes points in $\,\mathbb X\,$ beyond $\,\overline{\mathbb Y}\,$. This was already observed by Choquet~\cite{Choquet}, see also~\cite{AN}. Nevertheless, if (by chance)  for some homeomorphic boundary data the harmonic extension takes $\mathbb X$ onto $\mathbb Y$, then it remains injective in $\mathbb X$.

Harmonic  mappings have resulted from the outer variation of the Dirichlet integral, leading to the Lagrange-Euler equation. This equation  is not available when the energy integral is restricted to homeomorphisms; injectivity can be lost  upon the outer variation.

In different circumstances, Sobolev homeomorphisms
  are at the core of mathematical principles of Nonlinear Elasticity (NE) in which the Direct Method in the Calculus of Variations is the essential tool in finding the energy-minimal deformations.    It is from these perspectives  that one should look at the mappings $\,h \colon \overline{\mathbb X} \onto \overline{\mathbb Y}\,$   which are $\,\mathscr W^{1,2}\,$-weak limits of Sobolev homeomorphisms. If the target $\Y$ is a Lipschitz  domain, then such mappings are automatically uniform limits of homeomorphisms and, as such, become monotone. The concept of monotonicity is due to Morrey~\cite{Mor}.  By Morrey's definition, a continuous $\,h \colon \overline{\X} \onto \overline{\Y}\,$ (more generally, between any compact metric spaces) is monotone if every fiber $h^{-1} (y)$ of a point $y \in \overline{\Y} $ is connected in ${\overline{\X}}$. Consequently, as shown by  Whyburn~\cite{Wh} see also~\cite[p.138]{Whb}, the preimage of any connected set in $\overline{ \Y }$ is connected in ${\overline{\X}}$. Youngs' approximation theorem~\cite{Yo} tells us that all continuous monotone mappings $\,h: \overline{\mathbb X} \onto \overline{\mathbb Y}\,$ (in general, between 2D topological manifolds) are exactly the uniform limits of homeomorphisms $\, h_j \colon \overline{\mathbb X} \onto \overline{\mathbb Y}\,$.
  
 It is legitimate to perform the inner variation of the Dirichlet integral $\, \int_\mathbb X  |Dh(x)|^2 \textnormal d x\,$ subjected to monotone mappings $\,h \colon  \overline{\mathbb X} \onto \overline{\mathbb Y}\,$ of Sobolev class $\,\mathscr W^{1,2}(\mathbb X, \mathbb Y)$. This gives rise to the so-called \textit{Hopf-Laplace equation},
\begin{equation}\label{eq:hopf-laplace}
\frac{\partial }{\partial \bar z} (h_z \overline{h_{\bar z}}) =0 \,, 
\end{equation}
for $h \in \mathscr W^{1,2}(\mathbb X, \mathbb Y)$.  In ~\cite{Heb}  such solutions are called \emph{weakly Noether harmonic maps}.  We shall also discuss more general solutions  $\,h \in \mathscr W^{1,2}_{\textnormal{loc}}(\mathbb X, \mathbb C)\,$. This  places their \textit{Hopf product} $\,h_z \overline{h_{\bar z}}\,$ in $\mathscr L^1_{\textnormal{loc}} (\X)$, whose  Cauchy-Riemann derivative  $\frac{\partial }{ \partial \bar z} (h_z \overline{h_{\bar z}})\,$ is a Schwartz distribution. By Weyl's lemma $\,h_z \overline{h_{\bar z}}$ is in fact a holomorphic function.  We shall simply refer to them as the natural solutions of the  Hopf-Laplace equation. It is worth noting at this point that conformal change of the independent variable $\,z\in \mathbb X\,$ preserves the equation~\eqref{eq:hopf-laplace}. Thus we may assume, upon conformal transformation, that $\,\mathbb X \,$ is a unit disk. This observation explains why we shall not impose any regularity on $\,\mathbb X\,$,  except for being a Jordan domain. However some regularity of the target domain $\,\mathbb Y\,$ will be essential.

 It is clear that every harmonic mapping solves the Hopf-Laplace equation. Eells and Lemaire~\cite{EL2} inquired about the possibility of a converse result  for mappings with almost-everywhere positive Jacobian $J(z,h)=\det Dh (z) > 0$. For, if $h$ is $\mathscr C^2$-smooth the Hopf-Laplace equation is equivalent to $\,J(z,h)\, \Delta h =0\,$. The Eells-Lemaire question is seen to be false in general~\cite{Jost}. It may seem strange, but there exists a Lipschitz (actually piecewise orthogonal) mapping $\,\mathfrak h : \overline{\mathbb X} \into \mathbb R^2\,$ vanishing on $\,\partial \mathbb X\,$  whose Hopf product $\,\mathfrak h_z \overline{\mathfrak h_{\bar z}} = 0\,$,  almost everywhere (folding origami paper infinitely many times), see~\cite{IVV}.
 However, such  bizarre solutions do not occur  in the class of homeomorphisms; they turn out to be  harmonic mappings ~\cite{IKO8}.
Harmonic homeomorphisms are also known in the computer graphics literature~\cite{LPRM,RvBBWRF} under the name {\it least squares conformal mappings}. The message  is that without supplementary conditions of topological nature the general solutions to Hopf-Laplace equation are inadequate for  GFT and, certainly,  unacceptable in NE. The solutions that suit well for both purposes are \textit{\textbf{monotone Hopf-harmonics}}.

\begin{definition}  A continuous monotone mapping $\, h : \overline{\mathbb X} \onto \overline{\mathbb Y}\,$ of Sobolev class  $\, \mathscr W^{1,2}_{\loc}(\mathbb X, \mathbb C)\,$ which satisfies the equation (\ref{eq:hopf-laplace})  is called a \textit{monotone Hopf harmonic map}.
\end{definition}

In this class of mappings we gain, among other results,  an analogue of RKC-Theorem for non-convex targets. Let us first state one particular case, by assuming that the target domain $\,\mathbb Y\,$ is  $\mathscr C^2$-smooth.

\begin{theorem}\label{thm2}
   Given simply connected Jordan domains $\X$ and $\Y$, with   $\Y$ being $\mathscr C^2$-regular,  and an orientation-preserving\footnote{All given boundary homeomorphisms $g \colon \partial \X \onto \partial \Y$ are orientation-preserving without mentioning it explicitly.}  homeomorphism  $g \colon \partial \X \onto \partial \Y$  which admits a continuous  extension to $\,\mathbb X\,$ of Sobolev class $\,\mathscr W^{1,2}(\mathbb X, \mathbb R^2)\,$.   Then there exists a unique monotone Hopf-harmonic $h\colon \overline{\X} \onto \overline{\Y}$ of finite Dirichlet energy which agrees with $g$ on $\partial \X$.
\end{theorem}

A fundamental question arises:

\begin{question}\label{qu:main}
Let $\,\mathbb X\,,\,\mathbb Y \subset \mathbb C\,$ be bounded  simply connected domains  and $\,g \colon \partial \X \onto \partial \Y\,$  a monotone map. Does there exist a unique monotone Hopf-harmonic $\,h \colon \overline{\X} \to \mathbb C\,$ which coincides with $g$ on $\partial \X\,$? If that is the case, the equality  $\,h(\overline{\mathbb X}) = \overline{\mathbb Y}\,$ automatically holds.
\end{question}
 In such a generality this question seems to be over-committed. Nevertheless, the class of Lipschitz target domains (a standard assumption in NE) is wide enough to gain in interest.

\begin{theorem} [Existence] \label{thm:main}
Suppose that  $\X$ and $\Y$ are  simply connected Jordan domains, $\Y$ being  Lipschitz regular.  Let $g \colon \overline{\X} \onto \overline{\Y}$ be a homeomorphism of Sobolev class $\W^{1,2} (\X , \Y)$. Then there exists a  monotone Hopf-harmonic $h\colon \overline{\X} \onto \overline{\Y}$ of class $\W^{1,2} (\X , \Y)$ which agrees with $g$ on $\partial \X$.  Furthermore, $h$ is locally Lipschitz on $\X$ and a harmonic diffeomorphism  from  $h^{-1} (\Y)$ onto $\Y$.
\end{theorem}

This latter statement will be referred to as  \textit{partial harmonicity}. In particular, the set $\,\mathbb X \setminus h^{-1} (\Y) \,$ is squeezed into $\,\partial \mathbb Y\,$. The interpretation of partial harmonicity is that no continuum in $\,\mathbb X\,$ can be squeezed into a point in $\mathbb Y$. In other words, the interpenetration of matter may occur only in the regions adjacent to  $\,\partial\mathbb X\,$. 

\begin{remark} Speaking of the boundary homeomorphism $\,g \colon \overline{\X} \onto \overline{\Y}\,$ in Theorem \ref{thm:main}, it is certainly necessary to assume that $\,g\,$ admits a continuous finite energy extension to $\overline{\mathbb X}$;  harmonic extension is the one of smallest energy. However, if this assumption is made, there exists even a {homeomorphic} extension  $\,g : \overline{\mathbb X} \onto \overline{\mathbb Y}\,$ of Sobolev class $\,\mathscr W^{1,2}(\mathbb X, \mathbb Y)\,$ (of course, not necessarily harmonic).  This was shown in the recent work~\cite{KO}, in  which  the Lipschitz regularity of $\,\mathbb Y\,$ is essential. Curiously, the existence of finite energy harmonic extension depends only on the boundary map. Indeed, with the aid of a conformal transformation of $\,\mathbb X\,$ onto the unit disk $\,\mathbb D\,$, our boundary assumption reduces to the familiar Douglas condition \cite{Do}, formulated purely in terms of the map $g \colon \partial \mathbb D \onto \partial \mathbb Y$,
\begin{equation}\label{w12extension}
\int_{\partial \mathbb D} \int_{\partial \mathbb D}   \left| \frac{g(\xi) -g(\eta)}{\xi - \eta}\right|^2 \, \abs{\dtext \xi }\,  \abs{\dtext \eta } < \infty \, .
\end{equation}
\end{remark}

Our proof of Theorem~\ref{thm:main} expands on  the careful analysis of the structure of horizontal and vertical trajectories of the holomorphic quadratic Hopf differential $\, h_z \overline{h_{\bar{z}}} \,\, \textnormal d z\otimes \textnormal d z\,$, already initiated in \cite{IwaniecOnninenHopf, IOmono, IOsimply}.\\

\textit{Now comes the question of uniqueness}.  If $\,\mathbb Y\,$ is convex,  the unique harmonic extension of $\, g : \partial \mathbb X \onto \partial \mathbb Y\,$ is a homeomorphism of $\,\overline{\mathbb X}\,$ onto $\,\overline{\mathbb Y}\,$, by RKC theorem. Using an energy argument we shall see (Theorem \ref{Thm:main2} below) that this is the only monotone Hopf harmonic extension. 
The goal is to relax, as much as possible, the constraint of $\,\mathbb Y\,$  being convex. The following definition returns as its answer.
\begin{definition}[Somewhere Convexity] \label{PartialConvexity}
  A simply connected Jordan domain $\Y\subset \mathbb C\,$ is said to be {\it somewhere convex} if there is a disk $\, \mathbb D(y_\circ , \varepsilon)\,$ centered at a point $\, y_\circ \in \partial \mathbb Y\,$ and with radius $\,\varepsilon > 0\,$ whose intersection with $\,\overline{\mathbb Y}\,$ is convex.
\end{definition}
\begin{figure}[h!]
    \centering
    \includegraphics[width=0.99\textwidth]{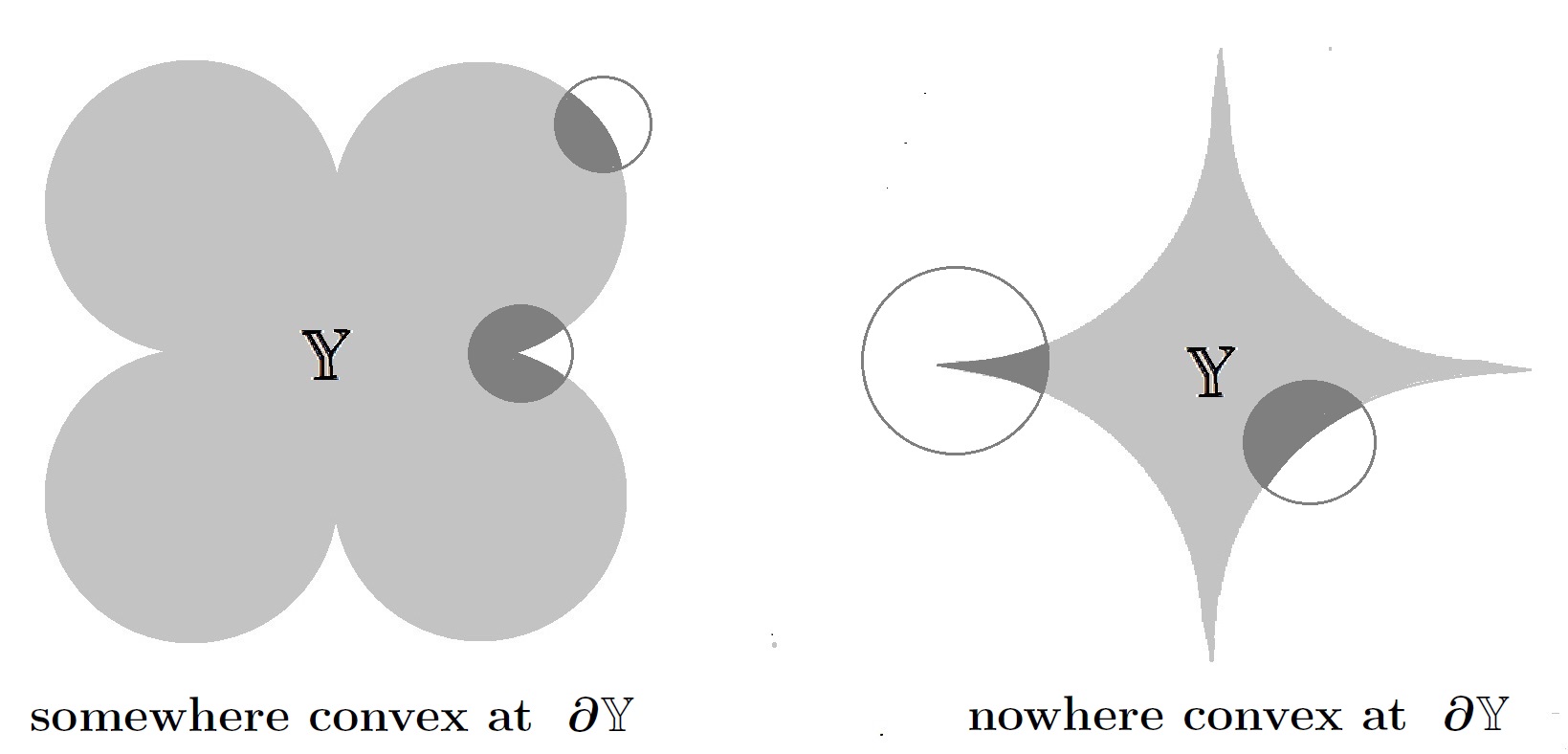}
     \caption{An illustration to Definition \ref{PartialConvexity}.}
\end{figure}
\begin{theorem} [Uniqueness] \label{Thm:main2} Under the assumptions in Theorem ~\ref{thm:main}, if in addition $\,\mathbb Y\,$ is somewhere convex, then the Hopf-harmonic map $\, h\colon \overline{\X} \onto \overline{\Y}\,$ is unique.
\end{theorem}

\textit{In summary}. Monotone Hopf harmonics  open a new area of study in GFT with applications to  the boundary value problems for hyper-elastic deformations of   plates (planar domains) and thin films (surfaces in $\mathbb R^3$). This is the way to explain in mathematical rigor the principle of non-interpenetration of matter in NE.  Topology of Monotone Sobolev mappings becomes  a new resource in nonlinear  PDEs.

\section{Prerequisites}
In this section we review from~\cite{Stb} useful concepts and results about Hopf differentials $h_z \overline{h_{\bar z}} \,\, \dtext z \otimes \dtext z$ and their trajectories. We, however, start with a powerful identity.

\subsection{An identity}
\begin{lemma}\label{lemmaidentity}
 Let $\X$, $\Y$  and $\mathbb G$ be bounded  domains in $\C$. Suppose that $h \colon \mathbb G \onto\mathbb Y$ and $H \colon \mathbb \X \onto \Y$ are orientation preserving $\mathscr C^\infty$-diffeomorphisms  of finite Dirichlet  energy.
Define
$\,f = H^{-1} \circ h \,\colon \mathbb G \onto \mathbb X$. Then we have
\begin{equation}\label{identity}
\begin{split}
\int_\X \abs{DH}^2  - \int_\mathbb G \abs{Dh}^2 & = 4 \int_{\mathbb G} \left[\frac{\abs{f_z-\sigma (z) f_{\bar z}}^2}{\abs{f_z}^2-\abs{f_{\bar z}}^2} -1 \right]
\, \abs{h_zh_{\bar z}}\, \dtext z\\
& + 4 \int_{\mathbb G} \frac{(\,\abs{h_z}\,-\,\abs{h_{\bar z}}\,)^2\cdot \abs{f_{\bar z}}^2
}{\abs{f_z}^2-\abs{f_{\bar z}}^2}\,\dtext z
\end{split}
\end{equation}

where
\[
\sigma = \sigma(z) = \begin{cases}
{h_z\overline{h_{\bar z}}}{\, \abs{h_z  \overline{h_{\bar z}}}^{-1}} \qquad &\textnormal{\;if } h_z\overline{h_{\bar z}}\ne 0 \\
0 & \textnormal{\;otherwise.}
 \end{cases}
\]

 The integrals in  (\ref{identity}) converge.
\end{lemma}
For the proof of Lemma~\ref{lemmaidentity} can be found in~\cite[Lemma 8.1]{IwaniecOnninenHopf}.

The following simply connected (not necessarily Jordan)  version of the Rad\'o-Kneser-Choquet
theorem will play a central role in our forthcoming arguments.
\begin{lemma}\label{lem:RKCsimply}
Consider a bounded simply connected domain $\mathbb U \subset \C$ and a bounded convex domain $\mathcal Q \subset \mathbb C$. Let $h \colon \partial \mathbb U \onto \partial \mathcal Q$ be a monotone mapping and $H \colon \mathbb U \to \C$ denote its harmonic extension. Then $H$ is a $\mathscr C^\infty$-diffeomorphism of $\mathbb U$ onto $\mathcal Q$.
\end{lemma}
The proof of this lemma we referee to~\cite{IOsimply}.

\subsection{Holomorphic quadratic differentials}
Let $\varphi (z)\, \dtext z \otimes \dtext z$ be a holomorphic quadratic differential in $\X$ with isolated zeros, called {\it critical points}. Through every noncritical point there pass two $\mathscr C^\infty$-smooth orthogonal arcs. A {\it vertical arc}
is a $\mathscr C^\infty$-smooth curve $\gamma= \gamma (t)$, $a < t < b$, along which
\begin{equation}
[\dot{\gamma} (t)]^2 \varphi \big( \gamma (t)\big) < 0 \, , \qquad \quad a<t<b\, .
\end{equation}
A {\it vertical trajectory} of $\varphi$ in $\X$ is a maximal vertical arc, that is, not properly contained in any other vertical arc. The {\it horizontal arcs} and {\it horizontal trajectories} are defined in an exactly similar way, via the opposite inequality. Through every noncritical  point of $\varphi$ there passes a unique vertical (horizontal) trajectory. A trajectory whose closure contains a critical point of $\varphi$ is called a {\it critical trajectory}. There are at most a countable   number of critical trajectories.

Every noncritical vertical trajectory $\gamma \subset \U$  in a simply connected domain $\U$ is a {\it cross cut}, see Theorem~15.1 in~\cite{Stb}.
\begin{lemma}\label{lem:strebel}
Consider a vertical arc $\gamma \subset \mathbb U$ in a simply connected domain $\mathbb U$. Let $\beta$ be any locally rectifiable curve in $\mathbb U$ which contains the endpoints of $\gamma$. Then
\begin{equation}\label{length}
\int_{{\gamma}} \abs{\varphi}^{\nicefrac{1}{2}}\, \abs{\dtext z} \le \int_{\beta} \abs{\varphi}^{\nicefrac{1}{2}}\, \abs{\dtext z},
\end{equation}
\end{lemma}
For the proof of this lemma we refer to~\cite[Theorem 16.1]{Stb}.

\begin{lemma}[Fubini-like integration formula]\label{Fublem}
Let $\varphi (z) \,  \dtext z \otimes \dtext z$  be a holomorphic quadratic differential in a simply connected domain $\U$, $\varphi \not\equiv 0$. Suppose that $F$ and $G$ are measurable functions in $\U$ such that
\begin{equation}\label{eqfub1}
\int_\U \abs{\varphi (z)} \abs{F(z)}\, \dtext z < \infty \quad \mbox{and} \quad \int_\U \abs{\varphi (z)} \abs{G(z)}\, \dtext z < \infty.
\end{equation}
Then for almost every vertical trajectory\footnote{The union of noncritical vertical trajectories has full $2D$ Lebesgue  measure in $\mathbb U$.} $\gamma$ of $\varphi (z) \dtext z \otimes \dtext z$, we have
\begin{equation}\label{eqfub2}
\int_\gamma \abs{\varphi (z)}^{\nicefrac{1}{2}} \abs{F(z)}\, \abs{\dtext z} < \infty \quad \mbox{and} \quad \int_\gamma \abs{\varphi (z)}^{\nicefrac{1}{2}} \abs{G(z)} \abs{\dtext z} < \infty.
\end{equation}\begin{itemize}
\item If
\begin{equation}\label{eqfub3}
\int_\gamma \abs{\varphi (z)}^{\nicefrac{1}{2}} {F(z)}\, \abs{\dtext z} =\int_\gamma \abs{\varphi (z)}^{\nicefrac{1}{2}} {G(z)}\, \abs{\dtext z},
\end{equation}
for almost every vertical trajectory $\gamma$ of $\varphi (z)\,  \dtext z \otimes \dtext z$  then
\begin{equation}\label{eqfub4}
\int_\U \abs{\varphi (z)} {F(z)}\, \dtext z= \int_\U \abs{\varphi (z)} {G(z)}\, \dtext z
\end{equation}
\item If
\begin{equation}\label{eqfub5}
\int_\gamma \abs{\varphi (z)}^{\nicefrac{1}{2}} {F(z)}\, \abs{\dtext z} \le \int_\gamma \abs{\varphi (z)}^{\nicefrac{1}{2}} {G(z)}\, \abs{\dtext z},
\end{equation}
 for almost every vertical trajectory $\gamma$ of $\varphi (z) \dtext z \otimes \dtext z$ then
\begin{equation}\label{eqfub6}
\int_\U \abs{\varphi (z)} {F(z)}\, \dtext z \le  \int_\U \abs{\varphi (z)} {G(z)}\, \dtext z
\end{equation}

\end{itemize}
\end{lemma}
Again, for the proof we refer to~\cite{Stb}. The following proposition follows from~\cite[Proposition 5.1]{CIKO}.

\begin{proposition}\label{pro:212}
Suppose that a monotone mapping $h \colon \overline{\X} \onto \overline{\Y}$ solves the Hopf-Laplace equation
\[h_z \overline{h_{\bar z}} = \varphi \, , \qquad \textnormal{where } \varphi \textnormal{ is holomorphic in $\X$.}\]
Then the preimage $h^{-1} (y_\circ)$ of a point $y_\circ \in \overline{\Y} $  is a continuum in $\overline{\X}$. If $h^{-1} (y_\circ)$ intersects a noncritical vertical trajectory of $\varphi (z) \, \dtext z \otimes \dtext z$, then it lies entirely in that trajectory.
\end{proposition}

Given a quadratic holomorphic differential $\varphi (z) \, \dtext z \otimes \dtext z$ we define two partial differential operators, called the {\it horizontal} and {\it vertical derivatives}
\begin{align*}
\partial_{_\mathsf H} = \frac{\partial}{\partial z} + \frac{\varphi}{\abs{\varphi}} \frac{\partial}{\partial \bar z} \qquad \mbox{ and } \qquad
\partial_{_\mathsf V} = \frac{\partial}{\partial z} - \frac{\varphi}{\abs{\varphi}} \frac{\partial}{\partial \bar z}.
\end{align*}
If $h$ satisfies the Hopf-Laplace equation $h_z \overline{h_{\bar z}}= \varphi$, then the horizontal and vertical trajectories of $\varphi (z) \dtext z \otimes \dtext z$ are the lines of maximal and minimal stretch for $h$. Precisely, the   following identities  hold.
\begin{align} & \abs{ \partial_{_\mathsf H} h } = \abs{h_z} + \abs{h_{\bar z}}, \qquad \abs{ \partial_{_\mathsf V} h } = \big|\abs{h_z} - \abs{h_{\bar z}}\big|  \label{hHhV} \\
& \abs{ \partial_{_\mathsf H} h } \cdot \abs{ \partial_{_\mathsf V} h }= \abs{J_h}, \qquad   \abs{ \partial_{_\mathsf H} h }^2-  \abs{ \partial_{_\mathsf V} h }^2=4 \abs{\varphi} \label{hHhVJh}
\end{align}
Here and after $J_h=\det Dh$.
As a consequence
\begin{equation}\label{becareful}
\abs{ \partial_{_\mathsf V} h }^2 \le \abs{J_h} \le \abs{ \partial_{_\mathsf H} h }^2.
\end{equation}

\begin{lemma}\label{lemtraj}
Let $\Omega$ be an open subset in $\C$ and $h \colon \Omega \to \C$ a locally Lipschitz solution of  the Hopf Laplace equation
\[h_z \overline{h_{\bar z}} = \varphi\, , \quad \textnormal{where } \varphi \textnormal{ is analytic in } \Omega \, . \]
Suppose that $J_h \equiv 0$ a.e. in $\Omega$. Then $h$ is constant on every vertical arc of the Hopf differential $\varphi (z)\, \dtext z \otimes \dtext z$.
\end{lemma}
\begin{proof}
Choose and fix a vertical arc, say
\[\gamma = \{ z(t) \colon a<t<b \, , \; \varphi (z(t)) \dot{z}^2 (t) <0 \textnormal{ and } \abs{\dot{z} (t)}\equiv 1 \}  \, . \]
{\bf Case 1.} We say that $\gamma$ is a ``good'' vertical arc if  for almost  every $t \in (\alpha, \beta)$  the mapping $h$ is differentiable at $z(t)$ and $J_h \big(z(t)\big) =0$. We begin with the chain rule along a ``good'' vertical arc,
\[ \frac{\dtext}{\dtext t} h \big(z(t) \big) =h_z \big( z(t) \big) \dot{z} (t) + h_{\bar z} \big( z(t)\big) \overline{\dot{z} (t)} \, .  \]
Hence
\[
\begin{split}
\Big| \frac{\dtext }{\dtext t}  h \big( z(t) \big) \Big|^2 &= \abs{ h_z \big( z(t) \big) }^2 + \abs{ h_{\bar z} \big( z(t) \big) }^2 \\
&+ h_z \big(  z(t) \big) \overline{h_{\bar z} \big(  z(t) \big)} \dot{z}^2 (t) + \overline{h_z \big(  z(t) \big) \overline{h_{\bar z} \big(  z(t) \big)} \dot{z}^2 (t)}
\end{split}
\]
Since $\gamma$ is a vertical arc the function defined by
\[\gamma (t) \bydef \varphi \big( z(t)\big) \dot{z}^2 (t)\]
is smooth real-valued and negative. Clearly, for almost every $\alpha < t < \beta$ we have  $ \varphi \big( z(t)\big) = h_z \big(z(t) \big) \overline{h_{\bar z}  \big(z(t) \big) }$ and
\[\abs{\gamma (t)} = \abs{ h_z \big( z(t) \big)  } \, \abs{ h_{\bar z} \big( z(t) \big)  } = \abs{ h_z \big( z(t) \big)  }^2  =  \abs{ h_{\bar z} \big( z(t) \big)  }^2 \, , \]
because the Jacobian determinant $J_h \big(z(t)\big) = \abs{ h_z \big( z(t) \big)  }^2 - \abs{ h_{\bar z} \big( z(t) \big)  }^2 $ vanishes.
We conclude with the equation
\[  \Big| \frac{\dtext }{\dtext t}  h \big( z(t) \big) \Big|^2 = \abs{\gamma (t)}  +  \abs{\gamma (t)} + \gamma (t) +\gamma (t) =0 \, \qquad \textnormal{for a.e. } t\in (\alpha, \beta) \,  .\]
 Hence $h \big( z(t)\big)$ is constant on $\gamma$.\\
 {\bf Case 2.} Now, let  $\gamma$ be an arbitrary vertical arc. It suffices to show that $h$ is locally constant on $\gamma$, say on $\gamma \cap \mathcal R$, where $\mathcal R$ is a curved rectangular box swept out by vertical arcs (as well as by  horizontal arcs). Upon a conformal change of variables, locally defined by the rule $\xi = \int \sqrt{\varphi (z)} \, \dtext z$, we see that $\mathcal R$ becomes an Euclidean rectangle, denoted by $\mathcal R^\ast$. The  vertical and horizontal arcs of $\varphi (z)\, \dtext z \otimes \dtext z$ become vertical and horizontal straight segments of $\mathcal R^\ast$, respectively. The new function $h^\ast (\xi)  \bydef h \big( z(\xi)\big)$ gives rise to the Hopf quadratic differential on $\mathcal R^\ast$
 \[ \varphi^\ast (\xi)  \,   \dtext \xi \otimes \dtext \xi\, , \qquad \textnormal{where }  \varphi^\ast (\xi)=  h^\ast_\xi \,  \overline{h^\ast _{\bar \xi}}  \]
whose trajectories are the vertical and horizontal segments. Also, $J_{h^\ast} (\xi) =0$ for almost every $\xi \in \mathcal R^\ast$. By Fubini's theorem almost every vertical segment is a ``good'' vertical arc of the differential  $\varphi^\ast (\xi) \,    \dtext \xi \otimes \dtext \xi$. By Case 1., $h^\ast$ is constant on almost every vertical segment of $\mathcal R^\ast$. Finally, since $h^\ast$ is continuous, it is constant on every vertical segment. This means that $h$ is constant on every vertical  arc in $\mathcal R$, as desired.
\end{proof}

\section{Proof of Theorem~\ref{thm:main}}
\subsection{Setting and notation}
Let $g \colon \overline{ \X} \onto \overline{  \Y}$ be given as in Theorem~\ref{thm:main}.
We denote the class of monotone  mappings $H \colon  \overline{\X} \onto \overline{\Y} $ in the Sobolev space $\W^{1,2}(\X , \C) $ which coincide with $g$ on $\partial \X$ by $\mathscr M_g (\overline{\X}, \overline{\Y})$. Furthermore, we write
\[\mathscr H_g (\overline{\X}, \overline{\Y}) = \{ H \in \mathscr M_g (\overline{\X}, \overline{\Y}) \colon H \colon \overline{\X} \onto \overline{\Y} \textnormal{ is a homeomorphism}\}\]
and
\[ \diff_g (\overline{\X}, \overline{\Y}) = \{ H \in \mathscr M_g (\overline{\X}, \overline{\Y}) \colon H \colon {\X} \onto {\Y} \textnormal{ is a diffeomorphism}\} \, .  \]
Clearly, $\mathscr H_g (\overline{\X}, \overline{\Y})$ is non empty, because it contains $g \colon  \overline{\X} \onto \overline{Y}$. Now, the direct method in the Calculus of Variations  reveals  that there always exists $ h \in \mathscr M_g\,(\overline{\mathbb X} , \overline{\mathbb Y})$ with smallest Dirichlet energy. Indeed,  the energy-minimizing sequence of monotone mappings in  $\mathscr M_g (\overline{\mathbb X} , \overline{\mathbb Y})$ converges weakly in  $\W^{1,2} (\X, \C)$ and it converges uniformly  to a monotone mapping $h \in \mathscr M_g\,(\overline{\mathbb X} , \overline{\mathbb Y }) $. The uniform convergence will follow from a general observation, see Remark~\ref{rem:equivcont}.

 Furthermore, the energy of $h$ equals exactly the infimum of the energy among all homeomorphisms in $\mathscr H_g (\overline{\X}, \overline{\Y})$. In symbols,
\begin{equation}\label{eq:exist}
\min_{ H \in \mathscr M_g(\overline{\mathbb X} , \overline{\mathbb Y}) } \int_\mathbb X |DH(x)|^2\,\textnormal d x = \inf_{ H \in \mathscr H_g (\overline{\mathbb X} , \overline{\mathbb Y}) } \int_\mathbb X |DH(x)|^2\,\textnormal d x \, .
\end{equation}
This follows from a Sobolev variant of Youngs' approximation theorem~\cite{IOmono}. Also, according to the approximation result~\cite{IKO2}, the infimum energy among diffeomorphisms leads to the same minimum value. Precisely, the equation~\eqref{eq:exist} extends as
\begin{equation}\label{eq:approx}
 \inf_{ H \in \mathscr H_g (\overline{\mathbb X} , \overline{\mathbb Y}) } \int_\mathbb X |DH(x)|^2\,\textnormal d x = \inf_{ H \in \diff_g(\overline{\mathbb X} , \overline{\mathbb Y}) } \int_\mathbb X |DH(x)|^2\,\textnormal d x  \, .
\end{equation}

\begin{remark}\label{rem:equivcont}
Every homeomorphism $g \colon \X \onto \Y$ between planar Jordan domains (not necessarily simply connected) admits a unique continuous extension as a map from $\overline {\X }$ onto $\overline{\Y}$, still denoted by $g \colon \overline{\X} \onto \overline{\Y}$. The extension is monotone. Also the boundary map $g \colon \partial \X \onto \partial \Y$ is monotone. Now consider a general monotone map $g \colon \overline{\X} \onto \overline{\Y}$ (not necessarily an extension of a homeomorphism) and assume that $\Y$ is Lipschitz regular; that is, locally $\partial \Y$ becomes a graph of a Lipschitz function upon suitable rotation. Then we have the following uniform bound of the modulus of continuity of every monotone map $g \in \W^{1,2} (\X, \R^2)$
\begin{equation}\label{eq:modofcont}
\abs{g(x_1)-g(x_2)}^2 \le C_{\X, \Y} \frac{\int_\X \abs{Dg(x)}^2  \, \dtext x}{\log \left(e + 1/\abs{x_1-x_2} \right)}
\end{equation}
for all $x_1, x_2 \in \overline{\X}$. Here the constant $C_{\X, \Y}$ depends only on the domains $\X$ and $\Y$, but not on the mapping $g$. The proof of~\eqref{eq:modofcont} can be found in~\cite{IOdef}. This estimate shows that a family of monotone mappings $g \colon \overline{\X} \onto \overline{\Y}$ which is bounded in $\W^{1,2} (\X, \R^2)$ is equicontinuous. In particular, every sequence in this family contains a subsequence converging uniformly and weakly in $\W^{1,2} (\X, \R^2)$ to a monotone map from $\overline{\X}$ onto $\overline{\Y}$ in the Sobolev class $\W^{1,2} (\X, \R^2)$.
\end{remark} 

\subsection{Existence}\label{sec:ex}
The existence of Hopf-harmonic monotone mapping $h$ in Theorem~\ref{thm:main} will be achieved by minimizing the Dirichlet-energy within the class $ \mathscr M_g (\overline{\mathbb X} , \overline{\mathbb Y}) $. First, note that  the existence of mapping with smallest Dirichlet-energy in $\mathscr M_g (\overline{\mathbb X} , \overline{\mathbb Y})$  follows from~\eqref{eq:exist}. Second,  the standard outer variation does not apply to this mapping.  But one can perform the inner variation, a change of variables in $\X$,
\[\frac{\dtext }{\dtext t} \bigg|_{t=0}  \mathcal E_{\X} [h \circ \eta_t] =0\]
Here  $\eta_t \colon \X \onto \X $ is a  family  of diffeomorphisms $\eta_t \colon \X \onto \X $ depending smoothly on the parameter $t\in \R$ which extend continuously up to  $\overline{ \X }$ as the identity map on $\partial \X$. The inner variation leads us to the claimed Hopf-Laplace equation~\cite{Job}~\cite[\S 3.1]{IwaniecOnninenHopf},
\[ \frac{\partial}{\partial \bar z} (h_z \overline{h_{\bar z}} ) =0\, ,   \qquad h \in \mathscr M_g (\overline{\mathbb X} , \overline{\mathbb Y}) \, . \]

\subsection{Lipschitz Regularity}
The Lipschitz regularity follows from the work~\cite{IKOlip} which, among other things, tells us that a solution to the Hopf-Laplace equation~\eqref{eq:hopf-laplace} with non-negative Jacobian $J(x,h) \ge 0$, a.e., is a locally Lipschitz  mapping. The fact that a monotone mapping $h\in  \mathscr M_g (\overline{\mathbb X} , \overline{\mathbb Y})$ has  $J(x,h) \ge 0$, a.e., follows from  the approximation result in~\cite{IOmono}. Indeed, there exists a sequence of diffeomorphims $h_j\in \diff_g (\overline{\mathbb X} , \overline{\mathbb Y})$ such that $h_j \to h$ in $\W^{1,2} (\X , \C)$.  Now, $J(x, h_j) \ge 0$ because $g \colon \partial \X \onto \partial \Y $ is positively oriented. Combining this with the fact that $J(x,h_j) \to J(x,h)$ a.e. in $\X$, the claimed inequality $J(x,h) \ge 0$ follows.

\subsection{Partial harmonicity}\label{sec:ph} This term refers to the fact that $h$  restricted to $h^{-1} (\Y) \subset \X$ is a harmonic diffeomorphism. To see this we may assume that the Hopf product $h_z \overline{h_{\bar z}}= \varphi$ does not vanish identically for otherwise $h$ would be holomorphic in $\X$. This is immediately  from the estimate
\[\abs{h_{\bar z}}^2 \le \abs{h_z \overline{h_{\bar z}}} =0 \, .  \]

 Let $\mathbb D$ be any open convex subdomain in $\Y$, for instance any open disk and $\mathbb U = h^{-1} (\mathbb D)$. According to Lemma 2.8 and 2.9 in~\cite{IOmono} $\U$ is simply connected (not necessarily Jordan) and the boundary mapping  $h \colon \partial \mathbb U \onto \partial \mathbb  D$ is monotone. We appeal to a Rad\'o-Kneser-Choquet result for simply connected domains, see Lemma~\ref{lem:RKCsimply}. Accordingly, the harmonic extension of the boundary mapping $h \colon \partial \mathbb U \onto \partial \mathbb  D$  to $\mathbb U$, is $\mathscr C^\infty$-diffeomorphism of $\mathbb U$ onto $\mathbb D$, denoted by $H \colon \mathbb U \onto \mathbb D$. We will prove the opposite inequality,
\begin{equation}\label{eq:partialharm}
\E_\U [h]= \int_\mathbb U \abs{Dh}^2 \le \int_\mathbb U \abs{DH}^2  = \E_\U [H]\, .
\end{equation}
Before passing to the proof of this inequality let us show how it would imply the partial harmonicity of $h$.
Obviously,
\[  \int_\mathbb U \abs{DH}^2 \le \int_\mathbb U \abs{Dh}^2 \, . \]
This shows that $h=H$ in $\mathbb U$ and therefore $h$ is a harmonic diffeomorphism of $\mathbb U$ onto $\mathbb D$. This property applies to every disk $\mathbb D \subset \Y$ and, consequently, $h$ is a local diffeomorphism. On the other hand, the mapping $h$ being monotone,  is actually a global diffeomorphism from $h^{-1}({\Y})$ onto $\Y$.

\subsubsection{Proof of the inequality~\eqref{eq:partialharm}}\label{sec:php}
 The proof  is based on the following consequence of Lemma~\ref{lemmaidentity}.
 \begin{lemma}\label{lem:star}
Let $f= H^{-1} \circ h \colon \mathbb U \onto \mathbb U $ and $\varphi = h_z \overline{h_{\bar z}} $. Then we have
\begin{equation}\label{compine}
\begin{split}
 \E_\U [H]- \E_\U [h] & \ge   \frac{4}{\norm{\varphi}_{\mathscr L^1(\U)}}\,  \left[\int_\U \Abs{f_z- \frac{\varphi}{\abs{\varphi}} f_{\bar z}} \, \sqrt{\abs{\varphi}}\sqrt{\abs{\varphi \big(f  \big)}}   \right]^2 \\
 & - 4 \int_{\U} \abs{\varphi}.
\end{split}
\end{equation}
Here we assume that $\varphi\not\equiv 0$. The term $\frac{\varphi}{ \abs{\varphi}}$ is understood as equal to zero
at the points where $\varphi$ vanishes.
\end{lemma}
\begin{proof} By the approximation result in~\cite{IOmono}, there exist a sequence of diffeomorphisms $h^j \colon \mathbb U \onto \mathbb D$, converging to $h$ uniformly and  in $\W^{1,2} (\mathbb U, \mathbb C)$. Moreover, each $h^j$ extends continuously to $\overline{\mathbb U}$ with $h^j =h$ on $\partial \mathbb U$. Let $\mathbb U'$ be a compactly contained subdomain of $\mathbb U$. Write $f^j = H^{-1} \circ  h^j \colon  \mathbb U \onto \mathbb U$. Applying Lemma~\ref{lemmaidentity} we obtain
\begin{equation}\label{eq:34}
\begin{split}
\int_{f^j(\mathbb U')} \abs{DH}^2  - \int_{\mathbb U'} \abs{Dh^j}^2 & = 4 \int_{\mathbb U'} \left[\frac{\abs{f^j_z-\sigma^j (z) f^j_{\bar z}}^2}{\abs{f^j_z}^2-\abs{f^j_{\bar z}}^2} -1 \right]
\, \abs{h^j_z \, h^j_{\bar z}}\\
& + 4 \int_{\mathbb U'} \frac{(\,\abs{h^j_z}\,-\,\abs{h^j_{\bar z}}\,)^2\cdot \abs{f^j_{\bar z}}^2
}{\abs{f^j_z}^2-\abs{f^j_{\bar z}}^2} \, ,
\end{split}
\end{equation}
where
\[
\sigma^j = \sigma^j(z) = \begin{cases}
{h^j_z \,  \overline{h^j_{\bar z}}} {\; \abs{h^j_z  \, \overline{h^j_{\bar z}}}^{-1}} \qquad &\textnormal{\;if } h^j_z\overline{h^j_{\bar z}}\ne 0 \\
0 & \textnormal{\;otherwise.}
 \end{cases}
\]
Since $f^j$ are   sense-preserving diffeomorphisms, the last integral in~\eqref{eq:34} is  nonnegative,
\[ \int_{\mathbb U'} \frac{(\,\abs{h^j_z}\,-\,\abs{h^j_{\bar z}}\,)^2\cdot \abs{f^j_{\bar z}}^2
}{\abs{f^j_z}^2-\abs{f^j_{\bar z}}^2}\,\dtext z \ge 0 \, .
\]
We estimate the first integral  by H\"older's inequality,
\[
\int_{\mathbb U'} \frac{\abs{f^j_z-\sigma^j (z) f^j_{\bar z}}^2}{\abs{f^j_z}^2-\abs{f^j_{\bar z}}^2} \, \abs{h^j_z\, h^j_{\bar z}}\, \dtext z \ge \frac{\left(  \int_{\mathbb U'} \, \abs{f^j_z - \sigma^j f^j_{\bar z} }\sqrt{\abs{ h^j_z \, h^j_{\bar z} }} \, \sqrt{\abs{\varphi (f^j(z))} } \,  \dtext z \right)^2}{\int_{\mathbb U'} J(z, f^j)\,  \abs{\varphi (f^j(z))} \, \dtext z}
\]
The denominator is bounded from above, by the $\mathscr L^1$-norm of $\varphi$,
\[ \int_{\mathbb U'} J(z, f^j) \, \abs{\varphi (f^j(z))} \, \dtext z = \int_{f^j (\mathbb U')} \abs {\varphi} \le \int_{\mathbb U} \abs {\varphi} \]
Since $\mathbb U \supset f^j (\mathbb U)$ for sufficiently large $j$, we have
\begin{equation}\label{eq:1011}
\begin{split}
\E_\U [H]- \E_{\U'} [h^j] & \ge 4 \,   \frac{\left(  \int_{\mathbb U'} \, \abs{f^j_z - \sigma^j f^j_{\bar z} }\sqrt{\abs{ h^j_zh^j_{\bar z} }} \, \sqrt{\abs{\varphi (f^j(z))} } \,   \dtext z \right)^2}{ \int_{\mathbb U} \abs {\varphi}\,  \dtext z} \\ &   - 4\, \int_{\U'} \abs{h^j_zh^j_{\bar z}}\, \dtext z
 \end{split}
 \end{equation}
Next, we let $j \to \infty$. We may assume, passing to a subsequence if necessary, that $h^j_z$ and $h^j_{\bar z}$ converge almost everywhere to $h_z$ and $h_{\bar z}$, respectively. Since the sequence $f^j= H^{-1} \circ h^j \colon \U \onto \U$ is converging to $f$ uniformly and
in $\mathscr W^{1,2}(\U')$ on subdomains $\U' \Subset \U$, it follows that
\[  \abs{f_z^j - \sigma^j f^j_{\bar z}} \, \sqrt{\abs{h^j_zh^j_{\bar z}} } \; \to \;  \abs{f_z - \sigma f_{\bar z}} \sqrt{\abs{h_z h_{\bar z}}}  \qquad \textnormal{in } \mathscr L^1 (\U')\]
and
\[\sqrt{\abs{\varphi (f^j(z))} } \to \sqrt{\abs{\varphi (f(z))} }\, ,  \quad \textnormal{everywhere.}  \]
Combining these facts with~\eqref{eq:1011}, we conclude
\[\E_\U [H]- \E_{\U'} [h] \ge 4 \frac{\left[\int_{\U'}\left| f_z-\sigma f_{\bar z}\right| \,\sqrt{\abs{\varphi (z) }} \, \sqrt{\abs{\varphi \big(f (z) \big)}}   \, \dtext z\right]^2} {\int_{\U} \abs{\varphi (z)}  \,\dtext z}  -4\int_{\U'}\abs{\varphi}.
  \]
Finally, since  $\U'$ was an arbitrary compact subset of $\U$, Lemma~\ref{lem:star} follows.
\end{proof}

Now having Lemma~\ref{lem:star}, the inequality~\eqref{eq:partialharm} would follow provided we can show that
\begin{equation}\label{eq:1111}
\int_\U \Abs{f_z- \frac{\varphi}{\abs{\varphi}} f_{\bar z}} \, \sqrt{\abs{\varphi}} \,  \sqrt{\abs{\varphi \circ f }} \, \dtext z    \ge \int_\U \abs{\varphi} \, \dtext z
\end{equation}
\begin{proof}[Proof of~\eqref{eq:1111}]
For almost every vertical noncritical trajectory $\gamma$, the mapping $f$ is locally absolutely continuous on $\gamma$. Let $\hat{\gamma}$ be a maximal subarc of $\gamma$ which lies in $\U$ so its endpoints belong to $\partial \U$.
Now, the change of variable formula gives
\begin{equation}\label{eq:11112}
\int_{\hat{\gamma}}   \Abs{f_z- \frac{\varphi}{\abs{\varphi}} f_{\bar z}} \, \sqrt{\abs{\varphi \circ f }}  = \int_{\hat{\gamma}}   \abs{f_{_\mathsf V}} \sqrt{\abs{\varphi \circ f }}  = \int_{f ( \hat{\gamma})} \sqrt{\abs{\varphi}} \, .
\end{equation}
Applying Lemma~\ref{lem:strebel} to the curve $\beta = f(\hat{\gamma})$ we have
\[ \int_{f  (\hat{\gamma})} \sqrt{\varphi}  \ge  \int_{\hat{\gamma}} \sqrt{\varphi}  \, .  \]
Combining this estimate with~\eqref{eq:11112}, we obtain
\[\int_{\hat{\gamma}}   \Abs{f_z- \frac{\varphi}{\abs{\varphi}} f_{\bar z}} \,  \sqrt{\abs{\varphi \circ f }} \ge  \int_{\hat{\gamma}} \sqrt{\varphi}  \, . \]
Now, the claimed inequality~\eqref{eq:1111} follows from this by the Fubini formula of  integration, see~\eqref{eqfub5}--\eqref{eqfub6}.
\end{proof}
This also completes the proof of~\eqref{eq:partialharm} and proves partial harmonicity. In general $h^{-1} (\Y)$ may or may not touch the boundary of $\X$. It is exactly at this point the somewhere convexity of $\Y$ comes into play.
\begin{lemma}\label{lem:touch}
Suppose that $\Y$ is somewhere convex and $h \colon \overline{\X} \onto \overline{\Y}$ is a monotone Hopf-harmonic mapping. Then $h^{-1} (\Y)$ touches $\partial \X$ along an open arc. Precisely $\overline{h^{-1} (\Y)}$ contains an open arc of $\partial \X$.
\end{lemma}
\begin{proof}
Recall that the somewhere convexity of $\Y$ means that there is an open  disk $\mathbb D$ centered at $y_\circ \in \partial \Y$ so that the intersection $\mathbb D \cap \Y$  (called {\it boundary cell}) is  a convex set.  Denote it by $\mathcal Q= \mathbb D \cap \Y$. We introduce the so-called {\it sealed boundary cell} $\mathcal Q^+ \bydef \mathbb D \cap \overline{\Y}$. Thus $\mathcal Q^+= \mathcal Q \cup \mathcal C$, where $\mathcal C = \mathbb D \cap \partial \Y$ is an open arc in $\partial \Y$. Clearly, $\mathcal Q^+$ is a connected subset of $\overline{\Y}$. Consider the preimage of the sealed boundary cell
\[\mathcal U^+ \bydef h^{-1} (\mathcal Q^+) \subset \overline{\X} \, . \]
Since $h \colon \overline{\X} \onto \overline{\Y}$ is monotone, $\mathcal U^+$ is connected. Now we have $\mathcal U^+ = \mathcal U \cup \Gamma$, where $\mathcal U=h^{-1} (\mathcal Q^+) \cap \X$ is a simply connected domain and $\Gamma = h^{-1} (\mathcal Q^+) \cap \partial \X$ is an open arc in $\partial \X$. Moreover, the mapping $h \colon \partial \mathcal U \onto \partial \mathcal Q$ is monotone. We refer to~\cite{IOmono} for the proof of these topological facts. It should be emphasized that $h^{-1} (\mathcal Q)$ need not be equal to $\mathcal U$.

In much the same way as in the proof of partial harmonicity, we appeal to  the Rad\'o-Kneser-Choquet theorem for simply connected domain, see Lemma~\ref{lem:RKCsimply}. Accordingly, let $H \colon \mathcal U \onto \mathcal Q$ be the harmonic extension of the boundary mapping $h \colon \partial \mathcal U \onto \partial \mathcal Q$. Now, the proof of the inequality~\eqref{eq:partialharm} in \S\ref{sec:php} goes in similar lines, namely we obtain
\[\int_{\mathcal U} \abs{Dh}^2 \le \int_{\mathcal U} \abs{DH}^2 \]
and conclude that $h=H$ on $\mathcal U$. This amounts to saying that
\[h^{-1}({\Y}) \supset \mathcal U \textnormal{ which touches } \partial \X \textnormal{ along } \Gamma \, . \]
\end{proof}

Before proceeding  the uniqueness  of  Hopf-harmonic monotone mappings, a proof of Theorem~\ref{Thm:main2}, let us  give an equivalent characterization for maps in question. In Section~\ref{sec:ex} we showed that
a mapping $h \colon \overline{\X} \onto \overline{\Y}$ which minimizes the Dirichlet energy among Sobolev monotone mapping in $\mathscr M_g (\overline{\X}, \overline{\Y})$ is a Hopf-harmonic monotone mapping. Actually, the converse also holds.
\subsection{Monotone Hopf-harmonics are the energy minimizers}

\begin{proposition}\label{pro:412}
Let $\Y$ be a simply connected Lipschitz domain in $\mathbb C$ and $g \colon \overline{ \X }\onto \overline{\Y }$ be an orientation-preserving  homeomorphism  of a Sobolev class  $\W^{1,2} (\X, \C)$, defined on a Jordan domain $\X$. Then  $h\in \mathscr M_g (\overline{\X}, \overline{\Y})$  is Hopf-harmonic if and only if
\[\begin{split}
\int_\X \abs{Dh(x)}^2 \, \textnormal d x & = \min_{ H \in \mathscr M_g(\overline{\mathbb X} , \overline{\mathbb Y}) } \int_\mathbb X |DH(x)|^2\,\textnormal d x \\ & = \inf_{ H \in \diff_g (\overline{\mathbb X} , \overline{\mathbb Y}) } \int_\mathbb X |DH(x)|^2\,\textnormal d x  \, . \end{split} \]
\end{proposition}
Here, the last equality   follows from~\eqref{eq:exist} and~\eqref{eq:approx}.
\begin{proof}
Let $h \in \mathscr M_g (\overline{\X}, \overline{\Y}) $ be a Hopf-harmonic mapping. Then
\[h_z \overline{h_{\bar z}}  = \varphi \qquad \textnormal{for some holomorphic } \varphi  \not \equiv 0 \, . \]
Let $\mathbb G = h^{-1} (\Y)$. In view of partial harmonicity in Section~\ref{sec:ph}, the mapping $h \colon \overline{\X} \onto \overline{\Y}$ is  a harmonic diffeomorphism from $\mathbb G$ onto $\Y$. Let $H \in \diff_g (\overline{\X} , \overline{\Y})$. Define
\[f = H^{-1} \circ h \colon  \mathbb G \onto \X \, . \]
In view of Lemma~\ref{lemmaidentity}, we see that
\[\begin{split}
\E_\X [H]-\E_\mathbb G [h] &=  4 \int_{\mathbb G} \left[\frac{\abs{f_z-\frac{\varphi }{\abs{\varphi }} f_{\bar z}}^2}{\abs{f_z}^2-\abs{f_{\bar z}}^2} -1 \right]
\, \abs{\varphi }\, \dtext z\\
& + 4 \int_{\mathbb G} \frac{(\,\abs{h_z}\,-\,\abs{h_{\bar z}}\,)^2\cdot \abs{f_{\bar z}}^2
}{\abs{f_z}^2-\abs{f_{\bar z}}^2}\,\dtext z \\
&=  4 \int_{\mathbb G} \left[\frac{  \abs{\partial_{_\mathsf V} f }^2}{J_f} -1 \right]
\, \abs{\varphi }\, \dtext z \\
& + 4 \int_{\mathbb G} \frac{(\,\abs{h_z}\,-\,\abs{h_{\bar z}}\,)^2\cdot \abs{f_{\bar z}}^2
}{J_f}\,\dtext z
\end{split}
\]
Before going further, let us observe that
\[\int_{\X \setminus \mathbb G} \abs{Dh}^2= 4 \int_{\X \setminus \mathbb G} \abs{ \varphi } \, . \]
Indeed, since $J_h \ge 0$ a.e. in $\X$ and $h$ belongs to the Sobolev class $\W^{1,2} (\X, \C)$, it follows that $J_h =0$ a.e. in $\X \setminus \mathbb G = h^{-1}{\partial \Y}$. This is because $h(\X \setminus \mathbb G) \subset \partial \Y$ and  $\partial \Y$ has zero $2$-dimensional measure. Now, by~\eqref{becareful}, it follows that $\abs{\partial_{_\mathsf V} h }^2=0$ a.e. in $\X \setminus \mathbb G$. Therefore,
\[\abs{Dh}^2 = \abs{\partial_{_\mathsf V} h}^2 + \abs{\partial_{_\mathsf H} h }^2 =  \abs{\partial_{_\mathsf H} h }^2 = 4 \abs{\varphi} \qquad \textnormal{a.e. in }  \X \setminus \mathbb G \,  \]
by~\eqref{hHhVJh}.
The above estimates give
\begin{equation}\label{eq:312}\begin{split}
\E_\X [H]-\E_\mathbb X [h] &= 4 \int_{\mathbb G} \frac{  \abs{\partial_{_\mathsf V} f }^2}{J_f} \, \abs{\varphi } \, \dtext z- 4\, \int_\X
\, \abs{\varphi }\, \dtext z \\
& + 4 \int_{\mathbb G} \frac{(\,\abs{h_z}\,-\,\abs{h_{\bar z}}\,)^2\cdot \abs{f_{\bar z}}^2
}{J_f}\,\dtext z  \, .
\end{split}
\end{equation}
Since $f$ is an orientation-preserving mapping we may employ the trivial estimate
\begin{equation}\label{eq:syr412}\int_{\mathbb G} \frac{(\,\abs{h_z}\,-\,\abs{h_{\bar z}}\,)^2\cdot \abs{f_{\bar z}}^2
}{J_f}\,\dtext z \ge 0 \, . \end{equation}
Next, we estimate the first integral on the right hand side of~\eqref{eq:312}. By H\"older's inequality,
\begin{equation}\label{eq:4121} \int_{\mathbb G} \frac{  \abs{\partial_{_\mathsf V} f }^2}{J_f} \, \abs{\varphi } \, \dtext z \ge \frac{\left(\int_\mathbb G \abs{\partial_{_\mathsf V} f } \sqrt{\abs{\varphi}} \, \sqrt{\abs{\varphi \circ f}} \right)^2}{\int_\mathbb G  \abs{\varphi \circ f} J_f}\end{equation}
On the one hand,  changing  variables,  we see that the denominator equals
\begin{equation}\label{eq:4122}\int_\mathbb G  \abs{\varphi \circ f} J_f = \int_\X \abs{\varphi} \, . \end{equation}
Concerning the numerator, we shall make use of Fubini's theorem. First, we change the variables in line integrals over the vertical trajectories.
Namely, for almost every vertical noncritical trajectory $\gamma$ it holds that
\begin{equation}\label{eq:41222}
\int_\gamma  \abs{\partial_{_\mathsf V} f } \sqrt{\abs{\varphi \circ f}} \cdot  \chi_{_{\mathbb G}}= \int_{f (\gamma{_{|_{\mathbb G}}})} \sqrt{ \abs{\varphi} } \end{equation}

 Since $\varphi \in \mathscr L^1(\X)$   the trajectory $\gamma$  has two distinct endpoints $x_1, x_2$ on $\partial \X$, see~\cite{MS}.  By Lemma~\ref{lemtraj},  for almost every vertical trajectory $\gamma$  the mapping $h$ is constant on each component of $\gamma \cap (\overline{\X} \setminus \mathbb G)$. Therefore,    $f (\gamma{_{|_{\mathbb G}}})$ is a connected union of arcs and, as such,  is an arc itself. It has the same endpoints as $\gamma$.

\begin{figure}[h!]
    \centering
    \includegraphics[width=0.9\textwidth]{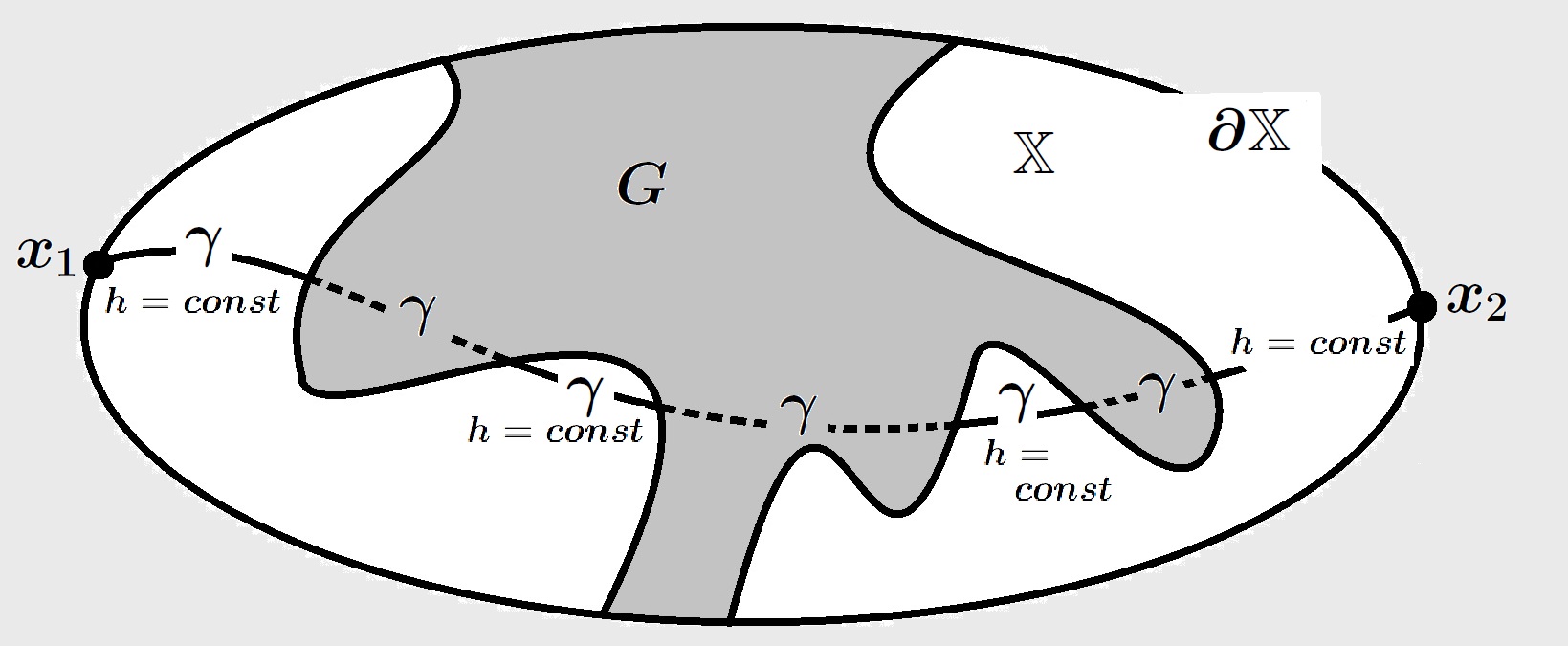}
    \caption{The set  $f(G\cap \gamma)$  is an arc}\label{fig:crosssection}
\end{figure}

 Now,  by~Lemma~\ref{lem:strebel} we have
\[ \int_{f ( \gamma{_{|_{\mathbb G}}})}  \sqrt{ \abs{\varphi} } \, \abs{\dtext z} \ge  \int_{\gamma} \sqrt{ \abs{\varphi} } \, \abs{\dtext z} \]
Therefore,
\[ \int_\gamma  \abs{\partial_{_\mathsf V} f } \sqrt{\abs{\varphi \circ f}} \cdot  \chi_{_\mathbb G} \ge \int_{\gamma} \sqrt{ \abs{\varphi} }   \]
 Fubini's ntegration formula~\eqref{eqfub6} yields
\begin{equation}\label{eq:4123} \int_\mathbb G \abs{\partial_{_\mathsf V} f } \sqrt{\abs{\varphi}} \sqrt{\abs{\varphi \circ f}} \ge \int_\mathbb X \abs{\varphi} \end{equation}
Combining~\eqref{eq:4121} and \eqref{eq:4123}, we obtain
\begin{equation}\label{eq:5121} \int_{\mathbb G} \frac{  \abs{\partial_{_\mathsf V} f }^2}{J_f} \, \abs{\varphi } \, \dtext z \ge \int_\mathbb X \abs{\varphi} \end{equation}
This together with~\eqref{eq:312} and~\eqref{eq:syr412}  gives
\[ E_\X [H]-\E_\mathbb X [h] \ge 0 \, , \]
as claimed.  This also finishes the proof of  Proposition~\ref{pro:412}.
\end{proof}

\subsection{Uniqueness, proof of Theorem~\ref{Thm:main2}} Let $h$ and $H$ be Hopf-harmonic monotone mappings from $\overline{\X}$ onto $\overline{\Y}$   which  coincides with $g$ on $\partial \X$. Therefore,
\[
\begin{split}
h_z \overline{h_{\bar z}} & = \varphi \qquad \textnormal{for some holomorphic } \varphi \\
H_z \overline{H_{\bar z}} & = \psi \qquad \textnormal{for some holomorphic } \psi
\end{split}
\]
We may assume that $\varphi \not\equiv 0 \not\equiv \psi$. By Proposition~\ref{pro:412} both mappings $h$ and $H$  minimize the Dirichlet energy subject to Sobolev monotone mapping in $\mathscr M_g (\overline{\X}, \overline{\Y})$. Let us consider the subdomains of $\X$,   $\mathbb G \bydef h^{-1}(\Y)$ and $\mathbb G_H \bydef H^{-1}(\Y)$. These are simply connected domains. In view of the partial harmonicity in~\eqref{eq:partialharm}, the mappings $h \colon \mathbb G \onto \Y$ and $H \colon \mathbb G_H \onto \Y $ are harmonic diffeomorphisms. Thus $f = H^{-1} \circ h \colon  \mathbb G \onto \mathbb G_H$ is an orientation preserving diffeomorphism. We denote the inverse of $f$ by $g=f^{-1} = h^{-1} \circ H \colon \mathbb G_H \onto \mathbb G$.
Fix a disk $\mathbb D \Subset \mathbb G_H$. There exists a sequence of diffeomorphisms $H_k \in \diff_g (\overline{\X}, \overline{\Y})$ converging to $H$ uniformly and in $\W^{1,2} (\X, \C)$. In analogy to $f$ and $g$ we define
\[ f^k \bydef H_k^{-1} \circ h  \colon \mathbb G \onto \X \quad \textnormal{ and } \quad  g^k \bydef h^{-1} \circ  H_k \colon \X \onto \mathbb G
\]
Since $H_k \colon \overline{\mathbb D} \to \Y$  converge uniformly to $H \colon \mathbb D \to H(\overline{\mathbb D})$, where $H(\overline{\mathbb D})$
is a compact subset of  $\mathbb Y$, there is a neighborhood $\mathbb V$ of $H(\overline{\mathbb D})$, compactly contained in $\Y$,  such that  $H_k(\overline{\mathbb D}) \subset \mathbb V$ for all sufficiently large $k$, say for $k \ge k_\circ$. Since $\overline{\mathbb V}$ is compact in $\Y$ the set $\mathbb F \bydef h^{-1} (\overline{\mathbb V})$ is compact in $\mathbb G$. Then we note that
\[g^k (\overline{\mathbb D}) =h^{-1} \big(H_k (\overline{\mathbb D})\big) \subset h^{-1} (\overline{\mathbb V}) = \mathbb F \, .\]
Furthermore, $g^k$ converges uniformly to $g= h^{-1} \circ H \colon \overline{\mathbb D } \to h^{-1} \big( H(\overline{\mathbb D }) \big)$. In view of~\eqref{eq:312}, it follows that
\begin{equation}\begin{split}
\E_\X [H_k]-\E_\mathbb X [h] &= 4 \int_{\mathbb G} \frac{  \abs{\partial_{_\mathsf V} f^k }^2}{J_{f^k}} \, \abs{\varphi } \, \dtext z- 4\, \int_\X
\, \abs{\varphi }\, \dtext z \\
& + 4 \int_{\mathbb G} \frac{(\,\abs{h_z}\,-\,\abs{h_{\bar z}}\,)^2\cdot \abs{f^k_{\bar z}}^2
}{J_{f^k}}\,\dtext z  \, .
\end{split}
\end{equation}
Applying~\eqref{eq:5121} with $f^k$ in place of $f$
\[\begin{split} \E_\X [H_k]-\E_\mathbb X [h]  & \ge  4 \int_{\mathbb G} \frac{(\,\abs{h_z}\,-\,\abs{h_{\bar z}}\,)^2\cdot \abs{f^k_{\bar z}}^2
}{J_{f^k}}\,\dtext z \\
& \ge  4 \int_{\mathbb F} \frac{(\,\abs{h_z}\,-\,\abs{h_{\bar z}}\,)^2\cdot \abs{f^k_{\bar z}}^2
}{J_{f^k}}\,\dtext z   \, .  \end{split} \]
Since $h $ is an orientation-preserving diffeomorphism on $\mathbb G$, we have $\abs{h_z}\,-\,\abs{h_{\bar z}} \ge c >0$ for every $z \in \mathbb F \Subset \mathbb G$ and a constant $c=c(\mathbb F)>0$
\[ \begin{split}\E_\X [H_k]-\E_\mathbb X [h]   &\ge 4 c^2 \int_{\mathbb F} \frac{ \abs{f^k_{\bar z}}^2}{J_{f^k}}\,\dtext z =   4 c^2 \int_{f^k(\mathbb F )}  \abs{g^k_{\bar w} (w)}^2 \, \dtext w \\ & \ge 4 c^2 \int_{\mathbb D}  \abs{g^k_{\bar w} (w)}^2 \, \dtext w
 \end{split} \]
Here we have made the substitution $z=g^k(w)$.

Letting $k \to \infty$ we find that $g^k_{\bar w} \to 0$ in $\mathscr L^2 (\mathbb D)$. Since $g^k \to g$ uniformly
on $\mathbb D$, we see that $g_{\bar w}=0$ on $\mathbb D$. But  $\mathbb D \Subset \mathbb G_H$ was arbitrary, so $g \colon \mathbb G_H \onto \mathbb G$ and  $f=g^{-1} \colon \mathbb G \onto \mathbb G_H$ are conformal.

Next, using Lemma~\ref{lem:touch}, we are going to show that $f(z)=z$. Here,  the  assumption that the part of $\partial \Y$ is convex is employed.   By~Lemma~\ref{lem:touch} we obtain that $\overline{h^{-1} (\Y)}$ contains an open arc, $\Gamma \subset \partial \X$.  Now, the conformal map $f  \colon \mathbb G \onto \mathbb G_H$ extends continuously to $\Gamma$. Since  $h(z) = H(z) $ on the boundary of  $\X$, we have  that $f(z)=z$ on $\Gamma$.  Finally, we appeal to a general fact that two holomorphic functions in $\mathbb G$, continuous on $\overline{ \mathbb G }$, are the same if they coincide on an arc of $\partial \mathbb G$. Therefore, $f(z)=z$ in $\mathbb G$, which means that $h (z) = H(z)$ for all $z\in \mathbb G$. Now, the holomorphic functions $\varphi = h_z \overline{h_{\bar z}} $ and  $ \psi = H_z \overline{H_{\bar z}}$ coincide in $\mathbb G$ and so
\[\varphi (z)  =  \psi (z) \qquad \textnormal{ for all } z \in \X \, .  \]
What remains is to argue that $h=H$ in $\X \setminus \mathbb G$. Note that $h$ and $H$ have the same vertical trajectories (because $\varphi \equiv \psi$). By Lemma~\ref{lemtraj} they are constant on every connected component (arc) of every every vertical trajectory. Since $h$ and $H$ coincide on the endpoints of these arcs, we conclude that $h \equiv H$ on $\X \setminus \mathbb G$, which completes the proof.
\subsection{Proof of Theorem~\ref{thm2}} Since every $\mathscr C^2$-regular domain $\Y$ is a somewhere convex Lipschitz domain Theorem~\ref{thm2} follows from Theorem~\ref{thm:main} and Theorem~\ref{Thm:main2}.

\section{Examples}
We will now  demonstrate, by way of illustration, of how the above results work for monotone Hopf harmonics $\,h : \mathbb X \onto \mathbb Y\,$ between domains with certain symmetries. In our first example the target $\Y \subset \C$ has the butterfly shape, with exactly one non-convex boundary point, see Figure~\ref{ButterflyTarget}.

\begin{example}\label{ex:butterfly}
We use the polar coordinates for $z$ in the closed unit disk $\overline{\mathbb D}$, $z= \rho e^{i \theta}$, $0 \le \rho \le 1$ and $0 \le \theta < 2 \pi$. Define $h \colon \overline{\mathbb D} \to \C$ by the rule:
\[ \begin{split}
h(\rho e^{i \theta})&= 2 \rho \left[\sqrt{\rho} \sin (\nicefrac{3}{2}\, \theta) + i \sin \theta  \right]= z- \bar z - i \left[ z^{\nicefrac{3}{2}}- \bar z^{\nicefrac{3}{2}} \right].
\end{split}
\]
\end{example}

This mapping is Lipschitz continuous with
\begin{equation}\label{exzbarz}
h_z = 1-\nicefrac{3}{2}\, i  \sqrt{z}, \qquad h_{\bar z} = -1+\nicefrac{3}{2}\, i  \sqrt{\bar z}.
\end{equation}
Moreover, its Hopf differential is holomorphic
\begin{equation}
h_z \overline{h_{\bar z}} = - \nicefrac{1}{4} \,  \left(4+  9z \right).
\end{equation}
Thus $h$ solves the Hopf-Laplace equation $\frac{\partial}{\partial \bar z} \left(h_z \overline{h_{\bar z}}\right)=0$. Concerning topological behavior, the  ray $\mathbb I = \{z \colon \im z =0 \mbox{ and } 0 \le \re z \le 1\}$  is squeezed  into  the origin, which is a boundary point of $\Y$. Outside of the ray, the mapping  $h$ is homeomorphism and it takes $\mathbb D \setminus \mathbb I$ as a harmonic diffeomorphism  onto the  domain $\Y$ , see Figure~\ref{ButterflyTarget}.

\begin{figure}[h!]
    \centering
    \includegraphics[width=0.99\textwidth]{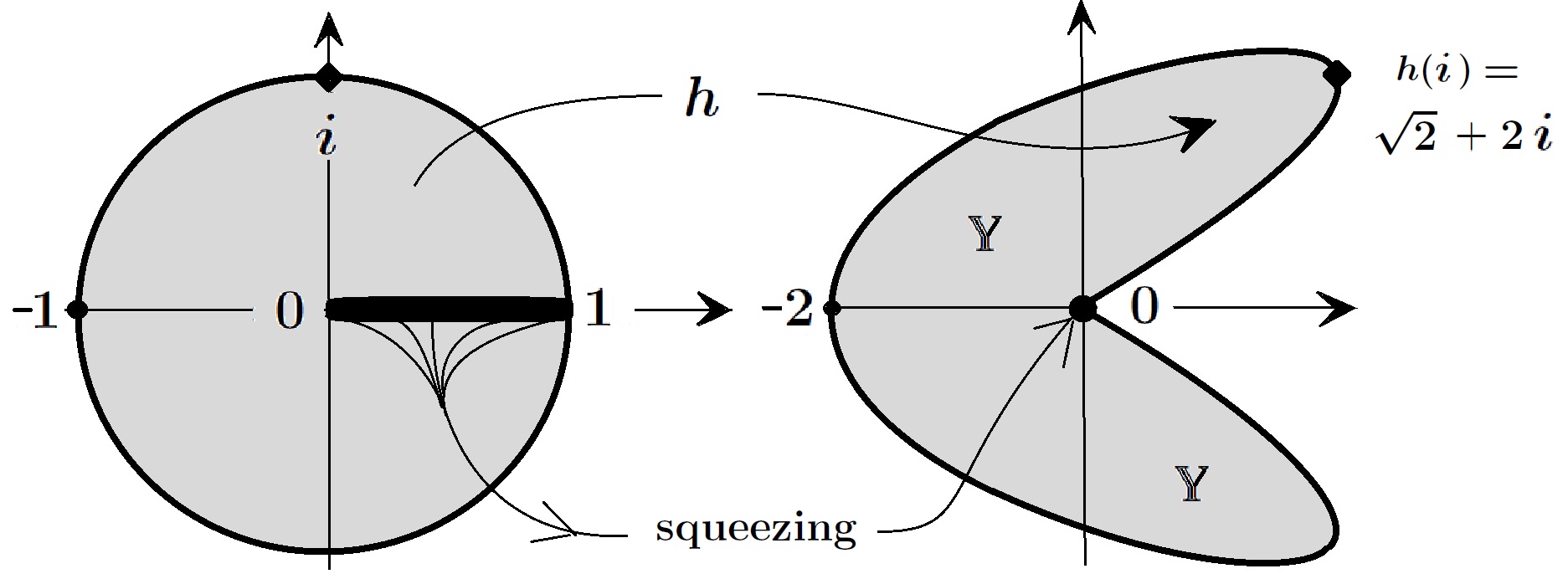}
    \caption{A horizontal segment is squeezed into a point where $\,\mathbb Y\,$ fails to be convex}\label{ButterflyTarget}
\end{figure}

\begin{example}\label{HammeringRectangle}
In our second example the target $\Y \subset \C$ is a semi-annulus in which the inner semi-circular boundary arc consists of non-convex points. Consider a horizontal strip
$$\,\mathcal S = \mathbb R \times \left[-\frac{\pi}{2} ,  \frac{\pi}{2} \right]  \; = \mathcal S_- \cup \mathcal S_+\;,\;\;
\textnormal{where} \;$$
$$\mathcal S_-  \bydef \left( -\infty , 0 \,\right ] \times \left[-\frac{\pi}{2} ,  \frac{\pi}{2} \right] \; \;\textnormal{and}\;\;\mathcal S_+ \bydef \left[\, 0 ,  +\infty \,\right ) \times \left[-\frac{\pi}{2} ,  \frac{\pi}{2} \right] $$

We define a mapping $\, h = u + i\, v  : \mathcal S \rightarrow \mathbb C\,$ by the rule

\[
h(x,y) = \begin{cases}
e^{i y} \cosh\, x \qquad &\textnormal{\;if } \; 0 \leqslant x <  +\infty \\
e^{i y} \qquad & \textnormal{\;if}\;  -\infty < x \leqslant 0
 \end{cases}
\]
It is straightforward to verify that $\,h\,$ is  a $\,\mathscr C^{1,1}\,$-smooth monotone Hopf harmonic, but not $\mathscr C^2\,$ -smooth. In fact, we have $\,h_z\, \overline{h_{\overline{z}}} \,\equiv -\frac{1}{4}\,$ in the entire strip. This map takes the vertical cross sections of $\,\mathcal S_+\,$ onto concentric semicircles $\,\mathcal C_\rho \bydef \{ (u,v) : \, u^2  + v^ 2 =  \rho^2 \;,\;  u \geqslant 0\,\}\,$\,, $\,1 \leqslant \rho < \infty\,$, see Figure~\ref{MapOfRectangle}. On the other hand, in $\,\mathcal S_-\,$ each  half line  $\,\{(x,y) ;\; -\infty < x \leqslant 0\,\}$, parametrized  by  $\,y \in \left[-\frac{\pi}{2},  \frac{\pi}{2} \right]\, $,  is squeezed into a point $\,e^{iy} \,\in \mathcal C_1\,$.
Now consider a rectangular box $\, \mathbb X =  \mathbb X_- \cup \mathbb X_+ \,$, where $$ \,\mathbb X_- \bydef  \left\{ (x,y) \colon \, -\ell < x < 0\;, \; -\frac{\pi}{2} < y < \frac{\pi}{2}\, \right\}\,$$  $$\,\mathbb X_+ \bydef \left\{ (x,y) \colon  \, 0 \leqslant x < T\;, \; -\frac{\pi}{2} < y < \frac{\pi}{2}\, \right\}\; $$
\end{example}
Our monotone Hopf harmonic map  $\,h\,$ takes $\,\mathbb X\,$ onto a semi-annulus $\, \mathbb Y \bydef \{ (u,v) \,;\, 1 < \sqrt{u^2 + v^2} < \cosh   T\;,\;  u > 0\,\}\,$, so $\,h^{-1}(\mathbb Y) = \mathbb X_+\,$.

\begin{figure}[h!]
    \centering
    \includegraphics[width=0.99\textwidth]{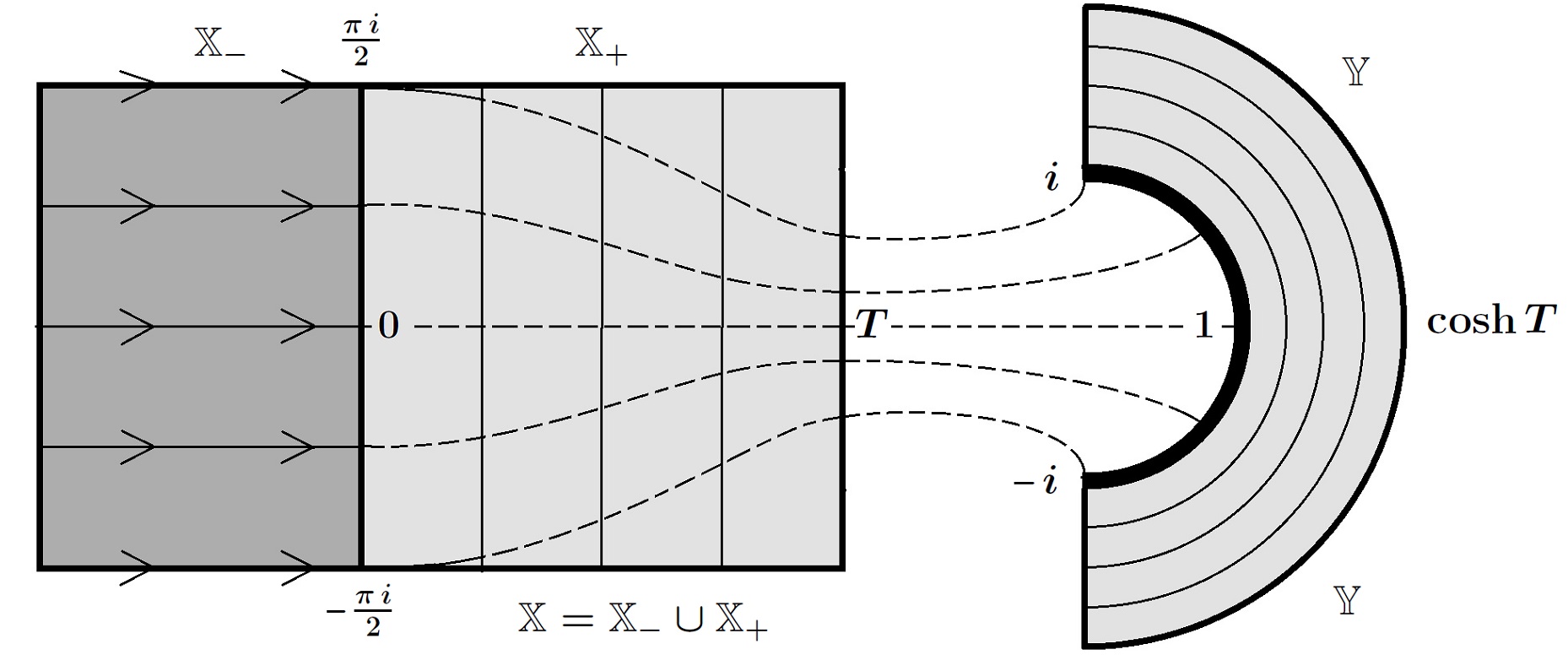}
    \caption{Points in $\mathbb X_-\,$ are projected onto the common boundary $\,\partial \mathbb X_- \cap \partial \mathbb X_+\,$,  and then transformed into the semicircular boundary arc of $\, \mathbb Y\,$, exactly where it fails to be convex.}\label{MapOfRectangle}
\end{figure}

\section{4-leaf clovers}

In our third example the target $\Y$ has a 4-leaf clovers shape.
\subsection{Circular and Elliptical Clovers} The reference configuration $\,\mathbb X \subset \mathbb C \simeq \mathbb R^2 \,$ will be  a union of four disks of radius 1 centered at the points $\,1, i, -1, -i\, $. Call $\,\mathbb X\,$ a \textit{circular 4-leaf clover}. Thus the boundary of $\,\mathbb X\,$ consists of four semicircular arcs, which we write as $\,\partial \mathbb X = \mathbf \Gamma_2\,\cup\, \mathbf \Gamma_{2i}\,\cup\, \mathbf \Gamma_{-2}\,\cup\, \mathbf \Gamma_{-2i}\,$. Each complex subscript here designates middle point of the arc, see Figure \ref{CircClover}.

\begin{figure}[h!]
    \centering
    \includegraphics[width=1.0\textwidth]{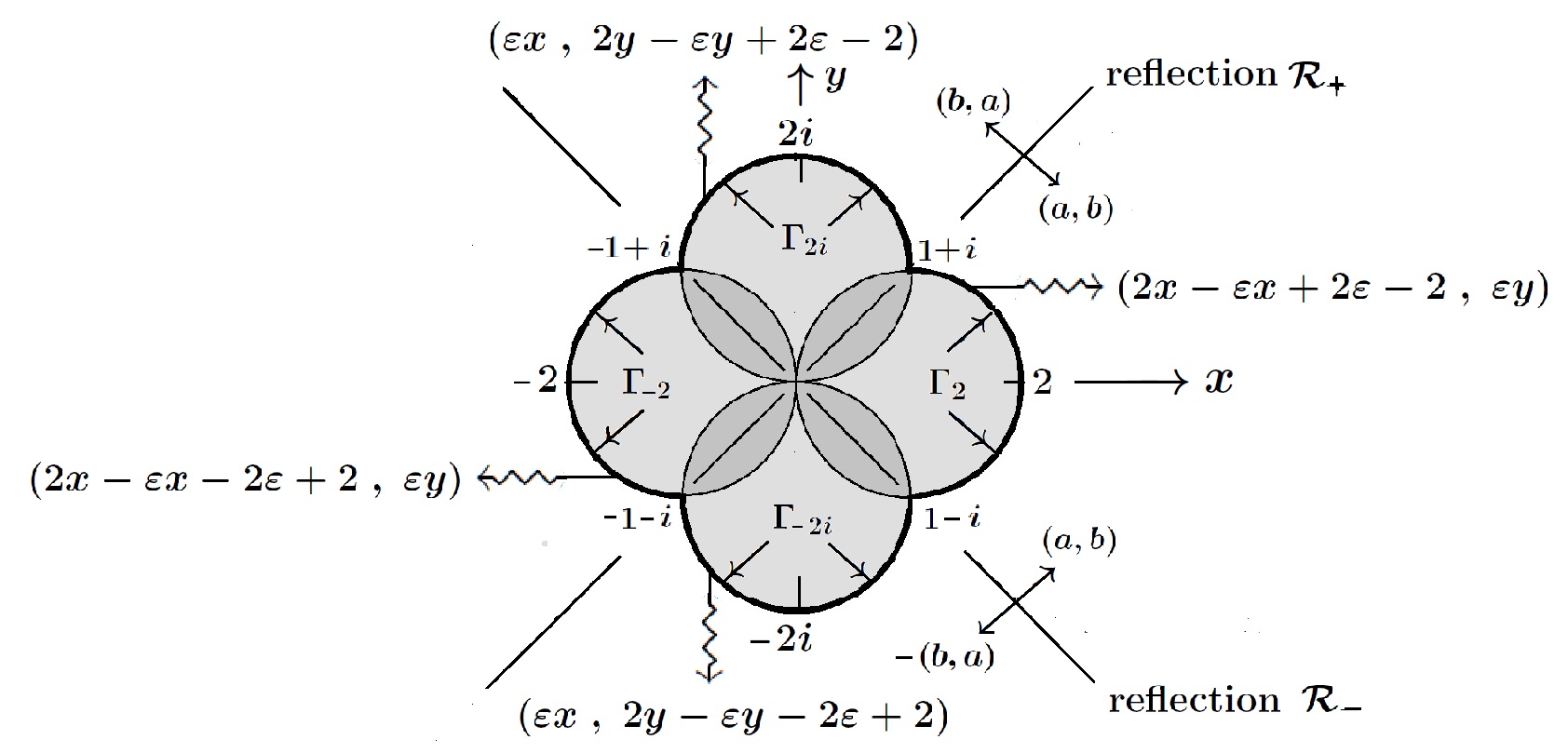}
    \caption{4-leaf circular clover and a piece-wise affine boundary data}\label{CircClover}
\end{figure}

The target domain $\,\mathbb Y \subset \mathbb C \simeq \mathbb R^2\,$ is a union of four ellipses obtained from the disks via affine transformations. We shall call it \textit{elliptical 4-leaf clover}, see Figure \ref{BoundaryMap}.

\begin{figure}[h!]
    \centering
    \includegraphics[width=1.0\textwidth]{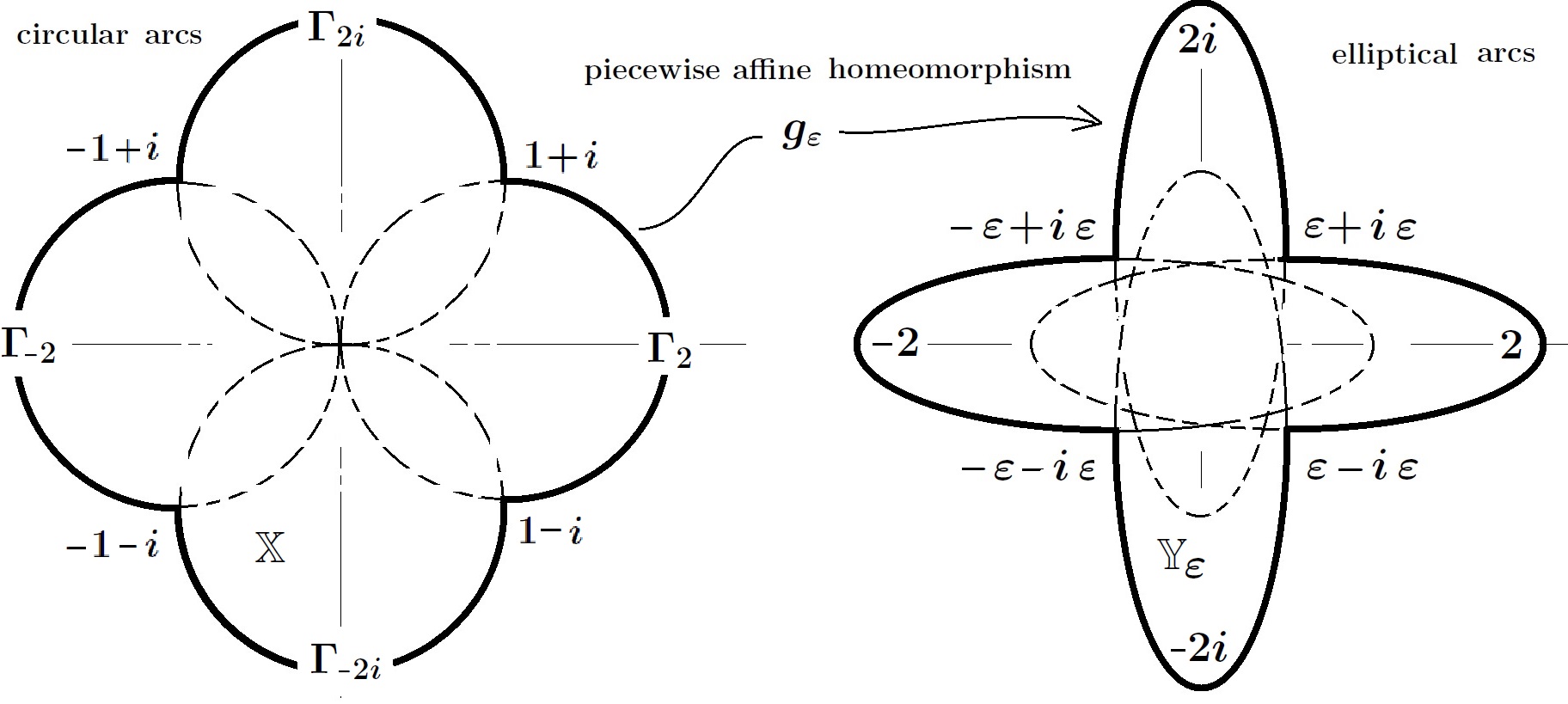}
    \caption{The boundary of the elliptical clover}\label{BoundaryMap}
\end{figure}

 The boundary of $\,\mathbb Y\,$ consists of four elliptical arcs, $\,\partial \mathbb Y \bydef g(\partial \mathbb X ) \bydef g(\mathbf \Gamma_2)\,\cup\, g(\mathbf \Gamma_{2i})\,\cup\, g(\mathbf \Gamma_{-2})\,\cup\, g(\mathbf \Gamma_{-2i})\,$, where $\,g : \partial \mathbb X  \onto \partial \mathbb Y\,$ is a piecewise affine map defined by the rule:

\[g(x,y) = g_\varepsilon(x,y) = \begin{cases}
( 2x - \varepsilon x + 2 \varepsilon  - 2\;, \; \varepsilon y  )\;, \;\textnormal{for} \; (x,y) \in \mathbf \Gamma_2 \;\\
( \varepsilon x\;,\; 2y - \varepsilon y + 2 \varepsilon  - 2 )\;, \;\textnormal{for} \; (x,y) \in \mathbf \Gamma_{2i} \;\\
(2x - \varepsilon x - 2 \varepsilon  + 2\;, \; \varepsilon y  )\;, \;\textnormal{for} \; (x,y) \in \mathbf \Gamma_{-2} \;\\
(\varepsilon x\;,\; 2y - \varepsilon y - 2 \varepsilon  + 2 )\;, \;\textnormal{for} \; (x,y) \in \mathbf \Gamma_{-2i} \;
\end{cases}\]
 Here $\,0 \leqslant \varepsilon \leqslant 1\,$ is a parameter to be chosen and fixed later on.  For now, the elliptical 4-leaf clover actually depends on $\,\varepsilon\,$, which we indicate by writing $\,\mathbb Y = \mathbb Y_\varepsilon\,$ when clarity requires it.

 \subsection{Harmonic Extension $\,G = G_\varepsilon\,$ }

 Except for $\,\varepsilon = 0\,$, the boundary map $\,g : \partial \mathbb X \onto \partial \mathbb Y\,$ is a homeomorphism.
 We see that $\,g_1(x,y) = (x,y)\,$, so  $\,\mathbb Y\, = \mathbb X\,$.  In this case the harmonic extension of $\,g_1\,$ is the identity on $\,\mathbb X\,$ as well. As one may have expected, when $\,\varepsilon\,$  drops below 1, but not too far (say $\, \varepsilon \in [ \varepsilon_\sharp ,   1 ]\,$ for some $\,0 < \varepsilon_\sharp \leqslant 1$),  the harmonic extension, denoted by $\,G = G_\varepsilon : \overline{\mathbb X} \into \mathbb R^2\,$ of the boundary data $\,g_\varepsilon :\partial \mathbb X  \onto \partial \mathbb Y_\varepsilon\,$ remains a diffeomorphism of $\,\mathbb X \onto \mathbb Y_\varepsilon\,$, see Figure \ref{CCL}.
 \begin{figure}[h!]
    \centering
    \includegraphics[width=1.0\textwidth]{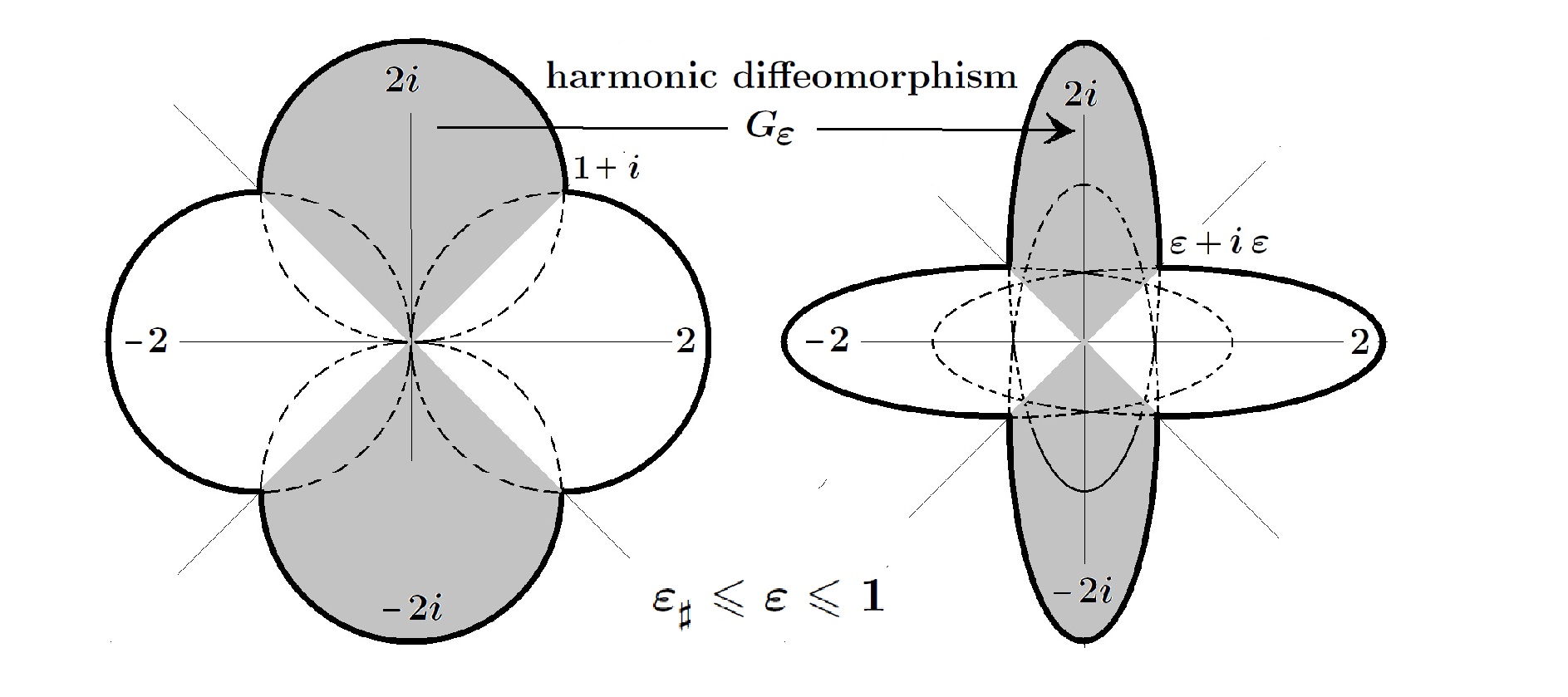}
    \caption{Circular clover and its diffeomorphic image by the harmonic extension of the boundary data}\label{CCL}
\end{figure}

\subsection{The Limit Case}

Let us take a quick look at the limit of  harmonic extensions as $\,\varepsilon \searrow 0\,$.

 In case  $\,\varepsilon = 0\,$ the 4-leaf clover degenerates to a cross of coordinate segments, see Figure \ref{LP}.

\[g_{_0}(x,y) =  \begin{cases}
( 2x   - 2\;, \; 0  )\;, \;\textnormal{for} \; (x,y) \in \mathbf \Gamma_2 \;\\
(0\;,\; 2y  - 2 )\;, \;\textnormal{for} \; (x,y) \in \mathbf \Gamma_{2i} \;\\
(2x   + 2\;, \; 0  )\;, \;\textnormal{for} \; (x,y) \in \mathbf \Gamma_{-2} \;\\
(0\;,\; 2y + 2 )\;, \;\textnormal{for} \; (x,y) \in \mathbf \Gamma_{-2i} \;
\end{cases}\]

\begin{figure}[h!]
    \centering
    \includegraphics[width=1.0\textwidth]{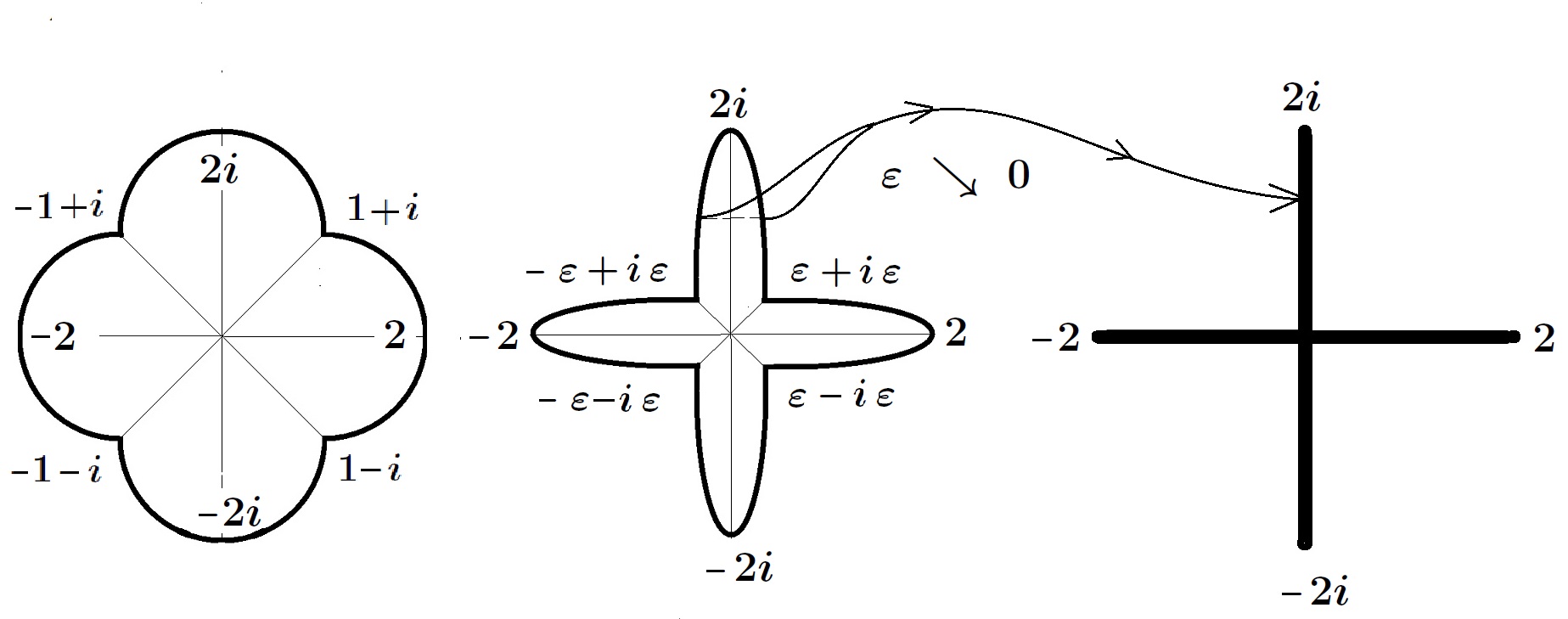}
    \caption{The uniform limit of the images of the boundary homeomorphisms degenerates to a cross of straight segments.}\label{LP}
\end{figure}

 We  always have the inclusion  $\, G_\varepsilon(\overline{\mathbb X})\, \supseteq \overline{\mathbb Y_\varepsilon}\,$; just because of continuity of $\, G_\varepsilon\,$. However, if $\,\varepsilon $ is small enough, we have even strict inclusion $\, G_\varepsilon(\overline{\mathbb X})\, \varsupsetneq \overline{\mathbb Y_\varepsilon}\,$. Indeed, suppose that, on the contrary, there is a sequence $\, \varepsilon_n \searrow 0\,$ for which  $\, G_{\varepsilon_n}(\overline{\mathbb X})\, \subset \overline{\mathbb Y_{\varepsilon_n}}\,$. The boundary homeomorphisms $\; g_{\varepsilon_n}  : \,\partial \mathbb X  \onto \partial\mathbb Y_{\varepsilon_n}\,$ converge uniformly to $\; g_{_0}  : \,\partial\mathbb X  \onto \partial \mathbb Y_0\,$. By the maximum/minimum principle it follows that $\, G_{\varepsilon_n} : \overline{\mathbb X} \onto \overline{\mathbb Y_{\varepsilon}}\,$  converge uniformly to a harmonic map $\,G_0 \bydef u + i v\,$ whose image $\,G_0(\overline{\mathbb X}) \,$ degenerates to a cross of straight line segments, see Figure \ref{LP}. Thus $\,u\cdot v \equiv 0\,$ on $\,\overline{\mathbb X}\,$. This is possible only when $\,u  \equiv 0\,$  or $\,v \equiv 0\,$, by the unique continuation property of harmonic functions, which is a contradiction.

\subsection{Critical Parameter $\,\varepsilon_\sharp\;$ }  We just have shown that there is so-called critical parameter $\,0 < \varepsilon_\sharp \leqslant 1\,$ such that:
whenever  $\,\varepsilon\,$ drops below $\,\varepsilon_\sharp\,$, the harmonic extension $\,G_\varepsilon :\overline{\mathbb X} \rightarrow \mathbb R^2\,$ of the boundary homeomorphism $\,g_\varepsilon : \partial \mathbb X \onto \partial \mathbb Y_\varepsilon\,$ takes part of $\,\mathbb X\,$ outside $\,\overline{\mathbb Y_\sharp}\,$,  as in Figure \ref{HFold}. Overlapping becomes inevitable.

\begin{figure}[h!]
    \centering
    \includegraphics[width=1.0\textwidth]{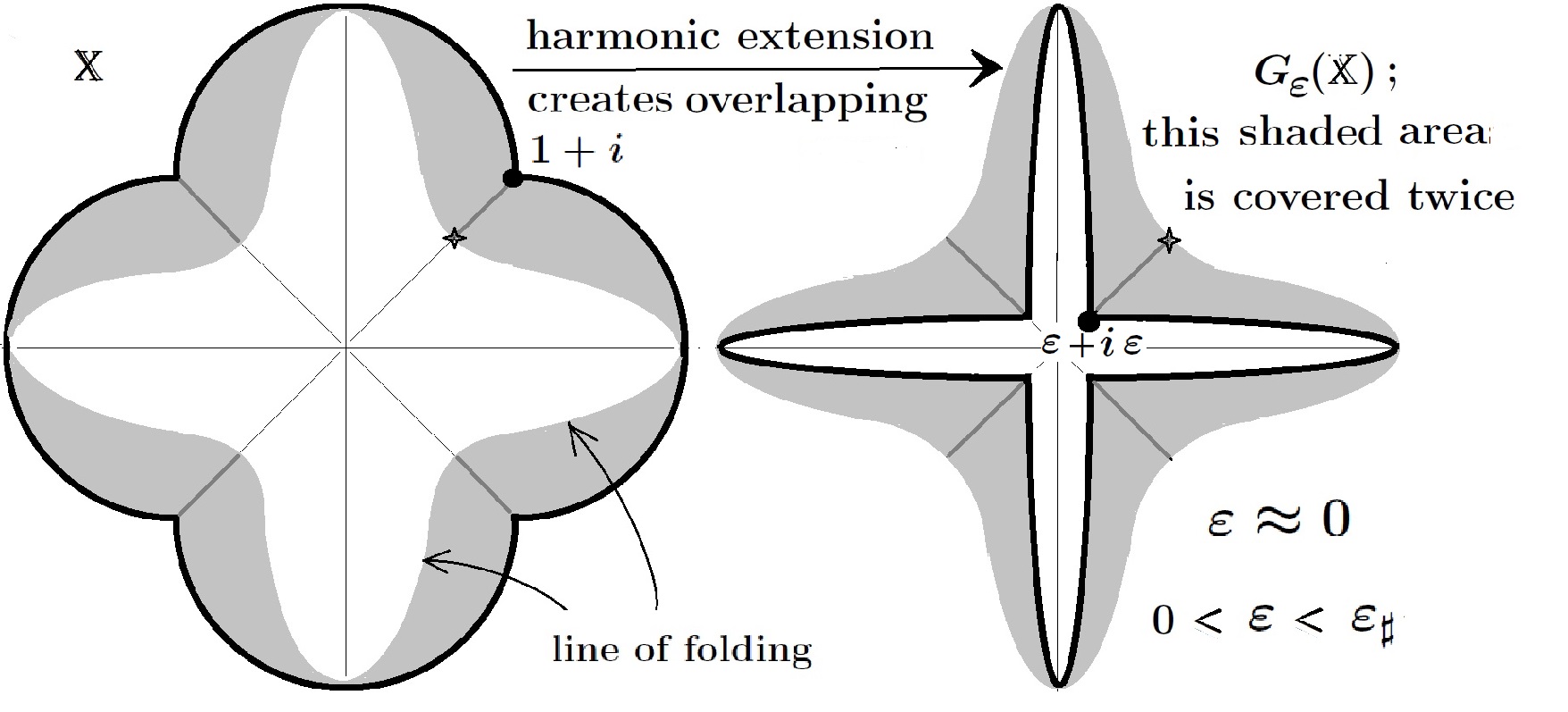}
    \caption{This hand made sketch may not be accurate regarding the actual lines of folding. }\label{HFold}
\end{figure}

In the mathematical models of Nonlinear Elasticity the overlapping is ruled out by the principle of non interpenetration of matter.  We just find ourselves forced to place topological restrictions on the mappings in question for minimizing the Dirichlet energy. Monotone Hopf harmonics turn out to be right solution; for, no overlapping may occur. As we shall illustrate in this example, monotone energy-minimal deformations will squeeze certain line fragments of $\,\mathbb X\,$ (emanating from $\, \partial \mathbb X\,$)  into non convex points of $\,\partial \mathbb Y\,$. Nevertheless Hopf harmonics, being limits of Sobolev homeomorphisms,  should take legitimate place in NE.
\subsection{Below the Critical Parameter}

\begin{center} \textit{This is the case $\,0 < \varepsilon < \varepsilon _\sharp\,$ when harmonic extensions fail.} \end{center}

From now on, we choose and fix a parameter $\,0 < \varepsilon < \varepsilon _\sharp\,$, so the harmonic extension $\,G_\varepsilon  : \mathbb X \into \mathbb R^2\,$ is ruled out by models of NE.
\subsection{Monotone Hopf Harmonic map $\,H = H_\varepsilon$}

Advantageously,  Theorem \ref{thm:main}, provides us with a unique monotone Hopf-harmonic map, denoted by $\, H = H_\varepsilon : \overline{\mathbb X} \onto \overline{\mathbb Y} =  \overline{\mathbb Y_\varepsilon}\,$,  of class $\,\mathscr C(\overline{\mathbb X}, \overline{\mathbb Y_\varepsilon})  \cap \mathscr W^{1,2}(\mathbb X, \mathbb Y_\varepsilon )\,$, which agrees with $\,g = g_\varepsilon\,$ on $\,\partial \mathbb X\,$. Furthermore, $\, H\,$ is a harmonic diffeomorphism from $\,H^{-1}(\mathbb Y)\,$ onto $\,\mathbb Y\,$.  Actually, among all monotone Sobolev mappings with prescribed boundary data $\,g : \partial \mathbb X \onto \partial \mathbb Y = \partial \mathbb Y_\varepsilon\,$, the map  $\,H\,$ is a unique one with smallest Dirichlet energy, see Proposition \ref{pro:412}. Our choice of 4-leaf clovers  comes from the fact that the symmetries of $\,\mathbb X\,$ and $\,\mathbb Y\,$ about the coordinate axes $\, y = 0 \,,\, x = 0\,$ and the diagonal lines  $\, y = x\,,\, y = -x\,$ will help us to locate the squeezing fragments of $\,\mathbb X\,$.\\

We start with the observation that the boundary data is also symmetric about these lines; in symbols,

\begin{equation}\begin{split}
 & g \circ \mathcal T_\pm = \mathcal T_\pm \circ g \;, \; \;\textnormal{where}\;\; \mathcal T_\pm (a, b) \bydef  \pm (a, -b)\,\;\;\textnormal{(respectively)}  \\
& g \circ \mathcal R_\pm = \mathcal R_\pm \circ g \;, \; \;\textnormal{where}\;\; \mathcal R_\pm (a, b) \bydef  \pm (b, a)\,\;\;\textnormal{(respectively)}
  \end{split}
\end{equation}
  The above commutation rules can easily be verified; make use of the explicit formulas conveniently provided in Figure \ref{CircClover} for this purpose.
  Using complex variable $\, z = x + i y\,$, the reflections $\,\mathcal T_\pm : \mathbb C \onto \mathbb C\,$ and $\, \mathcal R_\pm : \mathbb C \onto \mathbb C\,$ read as: $\,\mathcal T_\pm(z) = \pm \overline{z}\,$ and $\,\mathcal R_\pm(z) =  \pm i \overline{z}\,$. In particular, the boundary data is also invariant under rotation by right angle; namely, $\,(\mathcal T_\pm \circ \mathcal R_\pm) (z)  =  i z\,$.   The observed symmetries carry over to  the Hopf harmonic map $\,H: \overline{\mathbb X} \onto \overline{\mathbb Y}\,$ as well; precisely,

  \begin{equation}\label{SymmetryOfH} \begin{split}
  &H \circ \mathcal T_\pm = \mathcal T_\pm \circ H \;:\;  \overline{\mathbb X} \onto \overline{\mathbb Y} \;\;\textnormal{(respectively)}\\
&H \circ \mathcal R_\pm = \mathcal R_\pm \circ H \;:\;  \overline{\mathbb X} \onto \overline{\mathbb Y} \;\;\textnormal{(respectively)}
\end{split}
\end{equation}

To see this examine, in addition to   $\, H :\overline{\mathbb X} \onto \overline{\mathbb Y}\,$, four  monotone mappings;
\begin{equation}\label{eq:53} \begin{split}
&\,\mathbf T^\pm \bydef \mathcal T_\pm \circ H \circ \mathcal T_\pm\,:\,  \overline{\mathbb X} \onto \overline{\mathbb Y}\,\\
 &\,\mathbf R^\pm \bydef \mathcal R_\pm \circ H \circ \mathcal R_\pm\,:\,  \overline{\mathbb X} \onto \overline{\mathbb Y}\,
 \end{split}
 \end{equation}
  They all share the same  boundary data $\,g : \partial \mathbb X \onto \partial \mathbb Y = \partial \mathbb Y_\varepsilon\,$. Let their Hopf products be denoted by:
\begin{equation}\begin{split}
&\zeta(z) = H_z(z) \cdot \overline{H_{\overline{z}}(z)}\; \;,\;\; \textnormal{for}\;\; z \in \mathbb X\\
&\;\psi^\pm(z) = \mathbf T^\pm_z(z) \cdot \overline{\mathbf T^\pm_{\overline{z}}(z)} \;\\
&\;\; \phi^\pm(z) = \mathbf R^\pm_z(z) \cdot \overline{\mathbf R^\pm_{\overline{z}}(z)}
\end{split}
\end{equation}

 These functions are holomorphic in $\,\mathbb X\,$.  In fact, we have the following formulas for the Hopf products

$$
\psi^\pm(z)\; = \; \overline{\zeta(\pm \overline{z})} \;,\;\;\; \textnormal{and}\;\; \phi^\pm(z)\; = \; - \overline{\zeta(\pm i \overline{z})} \;,\;\;  \textnormal {respectively.}
$$
Since $\,\zeta\,$  is  holomorphic in $\, \mathbb X\,$, so are the Hopf products $\;\psi^\pm\,$ and $\,\phi^\pm\,$. Now comes the  uniqueness statement in Theorem \ref{thm:main}.  It tells us that all of the above five monotone mappings are the same. We just have established the commutation rules (\ref{SymmetryOfH}), whence it is readily inferred that $\,H\,$ takes points in each of the four lines of symmetry into the same line.
\subsection{Straight Line Segments of Symmetry}

To make it more precise, there are four straight line segments to be considered (sections of $\,\overline{\mathbb X}\,$ along the symmetry lines).

\[  \begin{cases}
 \mathbf A = \{ (x,0)\,;\,\;  -2 \leqslant x \leqslant 2\,\}\;\;,\;\;\textnormal{thus}\;\;\; H : \mathbf A \onto \mathbf A\\
\mathbf B = \{ (x,x)\,;\,\;  -1 \leqslant x \leqslant 1\,\}\;\;,\;\;\textnormal{thus}\;\;\; H : \mathbf B \onto \varepsilon \mathbf B\\
\mathbf C = \{ (0,y)\,;\,\;\,  -2 \leqslant y \leqslant 2\,\}\;\;,\;\;\textnormal{thus}\;\;\; H : \mathbf C \onto \mathbf C\\
\mathbf A = \{ (x,-x)\,;\, -1 \leqslant x \leqslant 1\}\;,\;\;\textnormal{thus}\;\; H : \mathbf D \onto \varepsilon \mathbf D
\end{cases}\]
In particular,  $\, H(0) = 0\,$.

\subsection{Janiszewski Theorem}
Our nearest goal is to show that:
\begin{lemma} \label{MonotoneJaniszewski} All the above four mappings  are monotone on their segments of definition.
\end{lemma}
\begin{proof}
The proof will only be given for the mapping $\,H : \mathbf A  \onto \mathbf A $; the other cases can be treated in much the same way.
The key ingredient is the topological theorem of  Z. Janiszewski  \cite{Jani} (1913).

\begin{definition}  With reference to  K. Kuratowski' book (\cite{Kuratowski}, Topology Vol. II, page 505),  the  Janiszewski space is a locally connected continuum having the following property:

If $\,\mathcal C_+\,$ and $\,\mathcal C_-\,$ are two continua whose intersection $\,\mathcal C_+ \cap \mathcal C_-\,$ is not connected, the union  $\,\mathcal C_+ \cup \mathcal C_-\,$ is a cut of the space (its complement is disconnected).
\end{definition}
\begin{center}\textit{The sphere $\,\mathbb S^2\,$ is a Janiszewski space} \end{center}
 see \cite{Kuratowski}, Ch. X, page 506.

\begin{figure}[h!]
    \centering
    \includegraphics[width=1.0\textwidth]{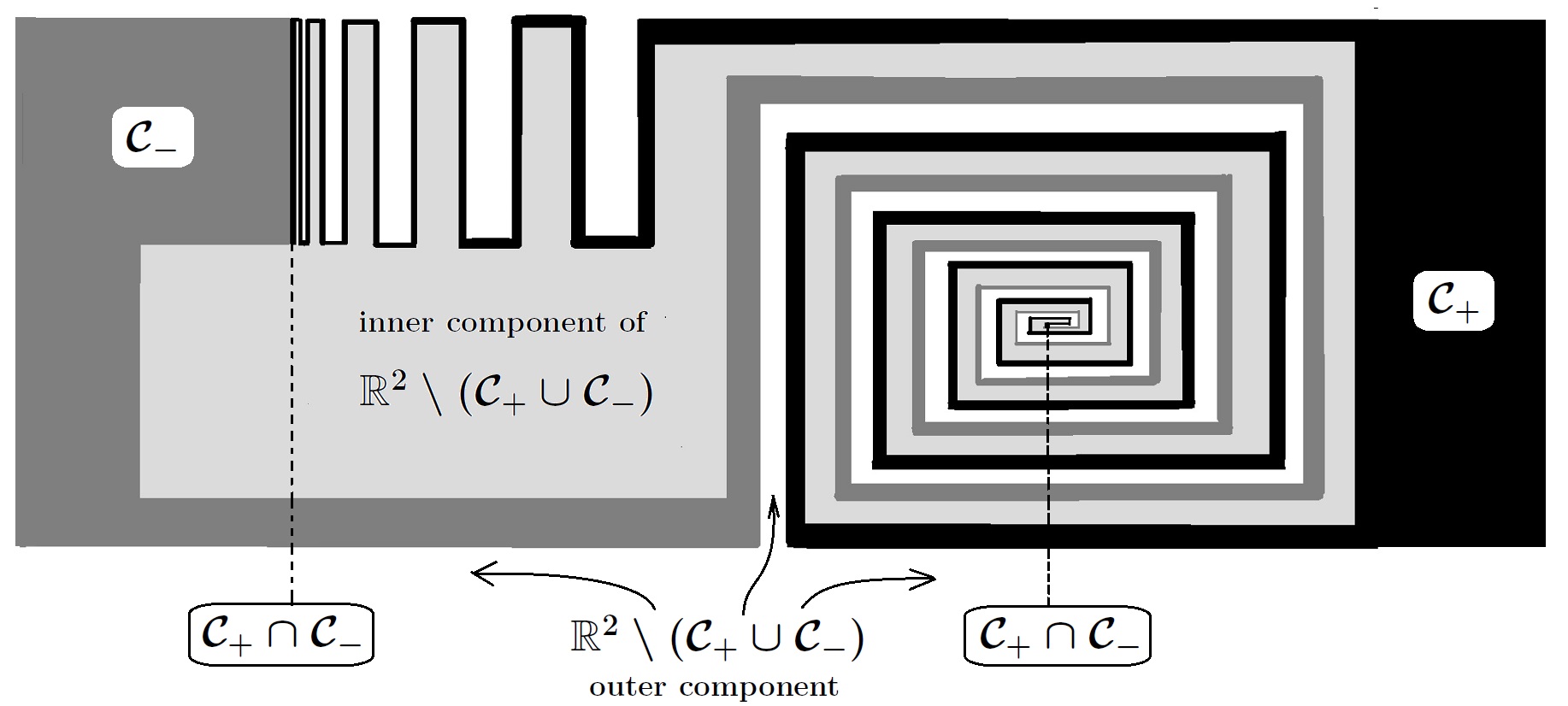}
    \caption{Janiszewski continua with two complementary components}
\end{figure}

 Now choose and fix a point in the target space, say $\,q \in H(\mathbf A) = \mathbf A\,$. Since $\,H : \overline{\mathbb X} \onto \overline{\mathbb Y} \,$ is monotone, its preimage $\,  \mathcal C \bydef \{ z \in \overline{\mathbb X} :   H(z) = q\,\}$ is a continuum. Our aim is to show that $\,\mathcal C \cap \mathbf A\,$ is connected. For this, we first observe (quite a general fact about monotone mappings) that $\,\mathcal C \subset \mathbb R^2\,$ is not a cut of $\,\mathbb R^2\,$, meaning that its complement is connected. Indeed, we have
 $$
 \mathbb R^2 \setminus \mathcal C  =  (\mathbb R^2 \setminus \overline{\mathbb X} ) \cup (\overline{\mathbb X} \setminus \mathcal C))\;,\;\;\textnormal{because}\;\; \mathcal C \subset \overline{\mathbb X}\, $$
  Both terms in this union are connected; the first by obvious reasons, the second is just a primage under $\,H :  \overline{\mathbb X} \onto \overline{\mathbb Y}\,$ of the connected set $\,\overline{\mathbb Y} \setminus \{q\}\,$. We need only verify that the intersection of those terms is not empty. But this is immediate from the formula
  $$
  (\mathbb R^2 \setminus \overline{\mathbb X} ) \cap (\overline{\mathbb X} \setminus \mathcal C)\; =\; \partial \mathbb X \setminus \mathcal C  \not = \emptyset
  $$
 We are now in a position to appeal to Janiszewski Theorem.

 For this, note that the above-mentioned symmetry of $\,H\,$  yields the respective symmetry of $\, \mathcal C\,$. Specifically, $\, z \in \mathcal C \;\;\rightleftarrows\;\; \bar{z} \in \mathcal C\,$. Then $\,\mathcal C\,$ can be decomposed in accordance with sign of $\,\Im m \,z \,$ as follows:   $\,\mathcal C =  \mathcal C_+  \cup \,\mathcal C_-\,$, where

 $$
  \mathcal C_+  = \{ z \in \mathcal C :\; \Im m \,z  \geqslant 0\,\} \;\;\textnormal{ and}\;\; \mathcal C_-  = \{ z \in \mathcal C :\; \Im m \,z  \leqslant 0\,\}
 $$
 It is readily seen that both $\,\mathcal C_+\,$ and   $\,\mathcal C_-\,$ are continua, and
 $$
 \mathcal C \cap \mathbf A   =   \mathcal C_+  \cap \mathcal C_-  \;\;\;
 $$

 Since $\,\mathcal C\,$ is not a cut of $\,\mathbb R^2\,$, by Janiszewski's Theorem,   the intersection  $\,\mathcal C_+  \cap \,\mathcal C_-\,$ must be connected, completing the proof of Lemma \ref{MonotoneJaniszewski}.
\end{proof}

\subsection{Segments of Squeezing}

The next step in our discussion is to look at the pre-images of the four points $\, \pm\varepsilon \pm i \varepsilon\,$  (exactly where  $\,\partial \mathbb Y_\varepsilon\,$ fails to be convex) under the monotone mappings $\,H : \mathbf B \onto \varepsilon \mathbf B\,$ and  $\, H : \mathbf D \onto  \varepsilon \mathbf D\,$, respectively. These pre-images, being connected, must be straight line segments in $\,\mathbf B\,$ and $\,\mathbf D\,$ with endpoints at $\,\pm 1   \pm i\,$, respectively. They do not pass through the origin, because  $\,H(0) = 0\,$. They have the same length (possibly zero) because of the rotational symmetry $\, H(iz) = i H(z)\,$. Let us denote these segments by,

$$
\mathbf B^+ = \{ t + i t\, ; \,  \rho \leqslant t  \leqslant 1\,\} \;\;,\;\; \mathbf B^- = \{ -t - i t\, ; \,  \rho \leqslant t  \leqslant 1\,\}
$$

$$
\mathbf D^+ = \{ t - i t \,; \,  \rho \leqslant t  \leqslant 1\,\} \;\;,\;\; \mathbf D^- = \{ -t + i t \,; \,  \rho \leqslant t  \leqslant 1\,\}
$$
\begin{remark} Note that at this stage of our arguments one cannot claim yet that $\, \mathbf B^\pm\,$ and $\,\mathbf D^\pm\,$ are the only collapsing sets, though it will turn out to be true.
\end{remark}

\subsection{Outside the Cracks}
 We now remove the collapsing segments $\,\mathbf B^\pm\,$ and $\,\mathbf D^\pm\,$ from $\,\mathbb X\,$ (interpreting  them as cracks in $\,\mathbb X\,$ that are squeezed to the boundary points at which $\,\partial \mathbb Y\,$ fails to be convex),

\begin{equation}\label{CutsInX}
\mathbb X_\divideontimes  \bydef \mathbb X \setminus ( \mathbf B^+  \cup \mathbf B ^- \cup \mathbf D^+ \cup \mathbf D^- )
\end{equation}

\begin{proposition} \label{PartHarm}The map $ \,H : \mathbb X_\divideontimes  \onto \mathbb Y\,$  is a harmonic diffeomorphism. In fact $\,\mathbb X_\divideontimes  =  H^{-1}(\mathbb Y)\,$.
\end{proposition}
\begin{proof} The proof is based on Proposition \ref{pro:412},  which asserts   that $\,H\,$   is the unique energy-minimal map among all monotone Sobolev mappings from $\, \overline{\mathbb X }\onto \overline{\mathbb Y}\,$ with the prescribed boundary data $\, g_\varepsilon  : \partial \mathbb X  \onto  \partial \mathbb Y_\varepsilon\,$. Our first aim is to construct a monotone Sobolev mapping $\, \widetilde{H }: \overline{\mathbb X} \onto \overline{\mathbb Y}\,$ whose energy does not exceed  the energy of $\,H\,$. For this purpose, we cut the circular clover $\,\overline{\mathbb X}\,$ into four sectors along the line segments $\,\mathbf B\,$ and $\,\mathbf D\,$. Let us introduce a  generic notation  for these sectors.

\[ \mathbb X_\sphericalangle \; \bydef \begin{cases}
 \mathbb X_1 \;\, \bydef    \{ (x , y )  \in \mathbb X ;\;  x > 0 , \;  - x < y < x\,\}\;\;,\;\;\textnormal{thus}\;\;\; 1 \in \mathbb X_1 \\
 \mathbb X_i \;\,\, \bydef    \{ (x , y )  \in \mathbb X ;\;  y > 0 , \;  - y < x < y\,\}\;\;,\;\;\textnormal{thus}\;\;\; i \in \mathbb X_i\\
 \mathbb X_{-1}  \bydef    \{ (x , y )  \in \mathbb X ;\;  x < 0 , \;   x < y < - x\,\}\;\;,\;\;\textnormal{thus}\;\;\; -1 \in \mathbb X_{-1}\\
 \mathbb X_{-i} \, \bydef    \{ (x , y )  \in \mathbb X ;\;  y < 0 , \;  y < x < -y\,\}\;\;,\;\;\textnormal{thus}\;\;\; -i \in \mathbb X_{-i}
\end{cases}\]
Analogously, we cut the elliptical clover into four sectors, Figure \ref{CCL}.

\[\mathbb Y_\sphericalangle\; \bydef \begin{cases}
 \mathbb Y_1 \;\, \bydef    \{ (x , y )  \in \mathbb Y ;\;  x > 0 , \;  - x < y < x\,\}\;\;,\;\;\textnormal{thus}\;\;\; 1 \in \mathbb Y_1 \\
 \mathbb Y_i \;\,\, \bydef    \{ (x , y )  \in \mathbb Y ;\;  y > 0 , \;  - y < x < y\,\}\;\;,\;\;\textnormal{thus}\;\;\; i \in \mathbb Y_i\\
 \mathbb Y_{-1}  \bydef    \{ (x , y )  \in \mathbb Y ;\;  x < 0 , \;   x < y < - x\,\}\;\;,\;\;\textnormal{thus}\;\;\; -1 \in \mathbb Y_{-1}\\
 \mathbb Y_{-i} \, \bydef    \{ (x , y )  \in \mathbb Y ;\;  y < 0 , \;  y < x < -y\,\}\;\;,\;\;\textnormal{thus}\;\;\; -i \in \mathbb Y_{-i}
\end{cases}\]

\subsection{Sector-wise RKC Extension of $\,H\,$}
  We first define  $\, \widetilde{H}\,$ on the boundary of each sector by setting $\, \widetilde{H} = H : \partial \mathbb X_\sphericalangle  \onto \partial \mathbb Y_\sphericalangle\,$ respectively.  In particular, $\, \widetilde{H} = H : \partial \mathbb X  \onto \partial \mathbb Y \,$. These boundary mappings are monotone. We extend them harmonically into the corresponding sectors, and denote by $\, \widetilde{H} :  \mathbb X_\sphericalangle  \onto  \mathbb Y_\sphericalangle \,$ respectively. It should be noted that these are the energy-minimal extensions. Moreover, by  Rad\'{o}-Kneser-Choquet Theorem, see Theorem  \ref{RKC}\,, each $\, \widetilde{H} :  \mathbb X_\sphericalangle  \onto  \mathbb Y_\sphericalangle \,$ is a homeomorphism, which  makes it clear that the map $\, \widetilde{H} :  \overline{\mathbb X}  \onto  \overline{\mathbb Y} \,$ so defined is  monotone and it lies in the Sobolev class $\, \widetilde{H} \in \mathscr W^{1,2} (\mathbb X, \mathbb Y)\,$. It is also important to notice the following formula for the domain $\,\mathbb X\,$ with cuts, as defined at (\ref{CutsInX}). Namely,

 \begin{equation}\label{NoSqueezedSubdomain}
\mathbb X_\divideontimes  \bydef \mathbb X \setminus ( \mathbf B^+  \cup \mathbf B ^- \cup \mathbf D^+ \cup \mathbf D^- )\;=\; \widetilde{H}^{-1} (\mathbb Y)
\end{equation}

 Proceeding further in this direction, we estimate the energy of $\,\widetilde{H}\,$ as follows:

 \begin{equation} \begin{split}
 &\mathscr E(\widetilde{H}_{| \mathbb X})  =  \mathscr E(\widetilde{H}_{| \mathbb X_1})\;+\; \mathscr E(\widetilde{H}_{| \mathbb X_i}) \;+\;\mathscr E(\widetilde{H}_{| \mathbb X_{-1}})\;+\; \mathscr E(\widetilde{H}_{| \mathbb X_{-i}}) \\
  &\leqslant  \mathscr E(H_{| \mathbb X_1})\;+\; \mathscr E(H_{|\mathbb X_i}) \;+\;\mathscr E(H_{| \mathbb X_{-1}})\;+\; \mathscr E(H_{| \mathbb X_{-i}}) \; =\;  \mathscr E(H_{| \mathbb X}) \nonumber
 \end{split}
 \end{equation}
On the other hand, according to Proposition \ref{pro:412}\,,  $\,H\,$ is the unique energy-minimal map among all monotone Sobolev mappings with the prescribed boundary data $\, g_\varepsilon :  \partial \mathbb X \onto \partial\mathbb Y\,$; $\,\widetilde{H}\,$  is thereby equal to $\,H\,$ in the entire region $\,\overline{\mathbb X}\,$. Formula (\ref{NoSqueezedSubdomain}) reads as:

\begin{equation}\label{NoSqueezedSubdomain2}
\mathbb X_\divideontimes  \bydef \mathbb X \setminus ( \mathbf B^+  \cup \mathbf B ^- \cup \mathbf D^+ \cup \mathbf D^- )\;=\; H^{-1} (\mathbb Y)
\end{equation}
The proof of Proposition \ref{PartHarm}  is completed by invoking the last statement of Theorem \ref{thm:main}, which tells us that $\,H\,$ is a harmonic diffeomorphism from $\,H^{-1}(\mathbb Y)\,$  onto $\mathbb Y\,$. For additional benefit, it also tells us that $\,H\,$ is locally Lipschitz on $\,\mathbb X\,$ (with cracks included).
\end{proof}
\subsection{Summary}

This example makes it clear that the \textit{Hopf Laplace equation} and monotonicity imposed on its solutions circumvent injectivity difficulties.
\begin{center}
{\it When harmonic extensions fail, \\
the Hopf-harmonics come to rescue.}
\end{center}

\begin{figure}[h!]
    \centering
    \includegraphics[width=1.0\textwidth]{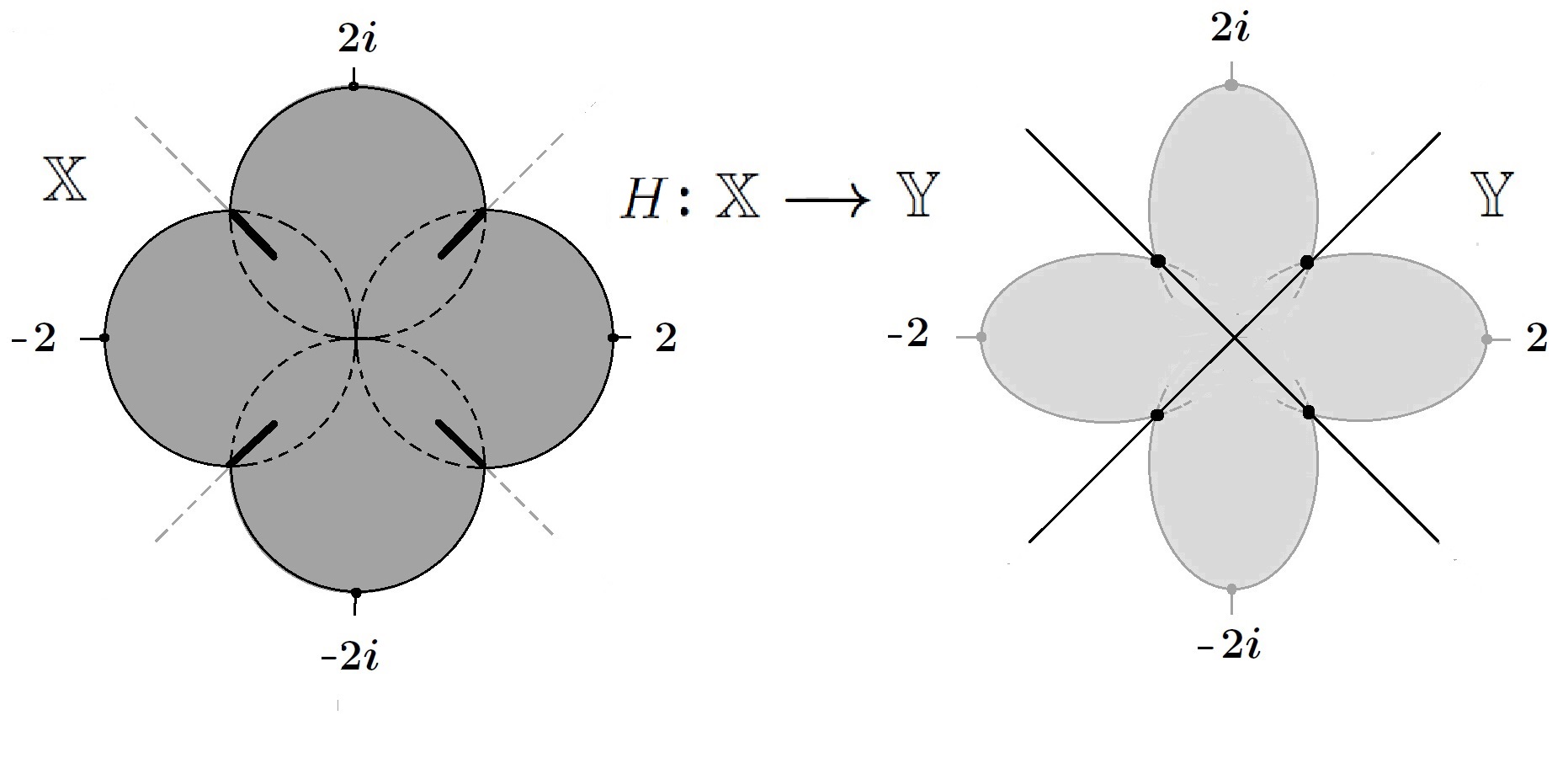}
    \caption{Cuts in a clover are inevitable when $\varepsilon \approx 0\,$. Finding an explicit  formula for  the length of cuts in terms of  $\,\varepsilon\,$, seemingly only a technical problem, is actually quite difficult.}\label{CutsClover}
\end{figure}

\section{An Alternating Process of \\Constructing Monotone Hopf-harmonics}
In this last section we set out a scheme of possible construction of   monotone Hopf-harmonic mapping of a simply connected Jordan domain $\,\mathbb B \subset \mathbb R^2\,$ onto a non-convex Lipschitz domain $\,\mathbb Y \subset \mathbb R^2\,$. The proposed scheme is motivated by the classical Schwarz Alternating Method that was originated in \cite{Schwarz1, Schwarz2, Schwarz3} for theoretical studies of conformal mappings and related planar harmonic functions.  More recently, this method gained a lot of attention as a very efficient algorithm for parallel computers.  There is a substantial literature on Schwarz Alternating Method for general second order elliptic PDEs,  beginning in 1951 with S.G. Mikhlin's paper \cite{M}  on convergence of the iterates. See the fundamental work of Lions~\cite{Lions1, Lions2, Lions3} for far reaching developments and the expository publications by Chan and Mathew~\cite{CM} and Le Tallec~\cite{Ta}, and the book of Smith, Bjorstad and Gropp~\cite{SBG}.

We do not attempt to rise and answer the most general questions. Our eventual aim here (not fully realized yet) is to illustrate that the idea of Schwarz remarkable technique can potentially be exercised for monotone solutions of the Hopf-Laplace equation.
 To emphasize the analogy and differences in our approach, let us take a glimpse of the Schwarz Alternating Method for constructing scalar (real valued)  harmonic functions. This scalar case reveals the first major difference; namely,  the comparison principle (a powerful tool for scalar harmonic functions)  is unavailable when studying complex harmonic homeomorphisms.\\

The classical Schwarz method works as follows. Let a domain $\,\mathbb B \subset \mathbb R^2\,$ be expressed as union of two overlapping subdomains $\,\mathbb B = \mathbb B_1 \cup \mathbb B_2\,$. We assume that for  each of these subdomains  one can solve the Dirichlet problem (under any reasonable boundary data).  Let a given (reasonable)  function $\, g \in\mathscr C(\overline{\mathbb B})\,$ represent a boundary data for the Dirichlet problem in $\,\mathbb B\,$. The alternating process begins with a function $\,g_1\,$ on $\,\overline{\mathbb B}\,$ that is harmonic on $\,\mathbb B_1\,$  and has the same values as  $\,g\,$ on $\,\partial \mathbb B_1\,$; call it \textit{ harmonic replacement} of $\,g \in \mathscr C(\overline{\mathbb B_1})\,$. On the remaining part $\,\overline{\mathbb B} \setminus \mathbb B_1\,$, we set $\,g_1 = g\,$. The next function $\,g_2 \in \mathscr C(\overline{\mathbb B})\,$ is harmonic on $\,\mathbb B_2\,$  with the same values as  $\,g_1\,$ on $\,\partial \mathbb B_2\,$, and coincides with $\,g_1\,$ on $\,\overline{\mathbb B} \setminus \mathbb B_2\,$. Continuing in this manner, we capture a sequence $\, \{g_1, g_2, g_3, g_4, ... \}\,$ which (under suitable geometric/analytic  hypotheses) converges to the solution of the Dirichlet problem in $\,\mathbb B\,$, see \cite{M}.

 The point to make here is that during this process the subdomains $\,\mathbb B_1\,$ and $\, \mathbb B _2\,$ stay the same for all time; only the boundary data of the harmonic replacements change. This remains in major contrast with our alternating approach for the monotone Hopf harmonics. Precisely, in our method the subdomains $\,\mathbb B_1\,$ and $\, \mathbb B _2\,$  will vary, but their images under the harmonic replacements will always be the same convex domains, say $\,\mathbb Y_1\,$ and $\,\mathbb Y_2\,$, respectively. We can make this clear by means of the following example.

 \subsection{An Example} It involves no loss of generality in assuming that $\,\mathbb B\,$ (a simply connected Jordan domain) is the unit disk. In our example,  the target $\,\mathbb Y \subset \mathbb R^2\,$ is  assembled with two convex subdomains $\,\Y_1\,$ and $\,\Y_2\,$ such that $\,\Y_1 \cap\Y_2 \,\not= \emptyset\,$; these composition of $\,\mathbb Y\,$  will stay the same during the entire alternating process. In particular, the target domain $\Y\bydef\Y_1 \cup \Y_2$ is somewhere convex. We shall also assume that $\,\mathbb Y\,$ is Lipschitz regular. Furthermore, taking for $\,\mathbb Y\,$  a symmetric heart shaped domain, as in  Figure \ref{CHeart1},  considerably eases the arguments.

\begin{figure}[h!]
    \centering
    \includegraphics[width=1.1\textwidth]{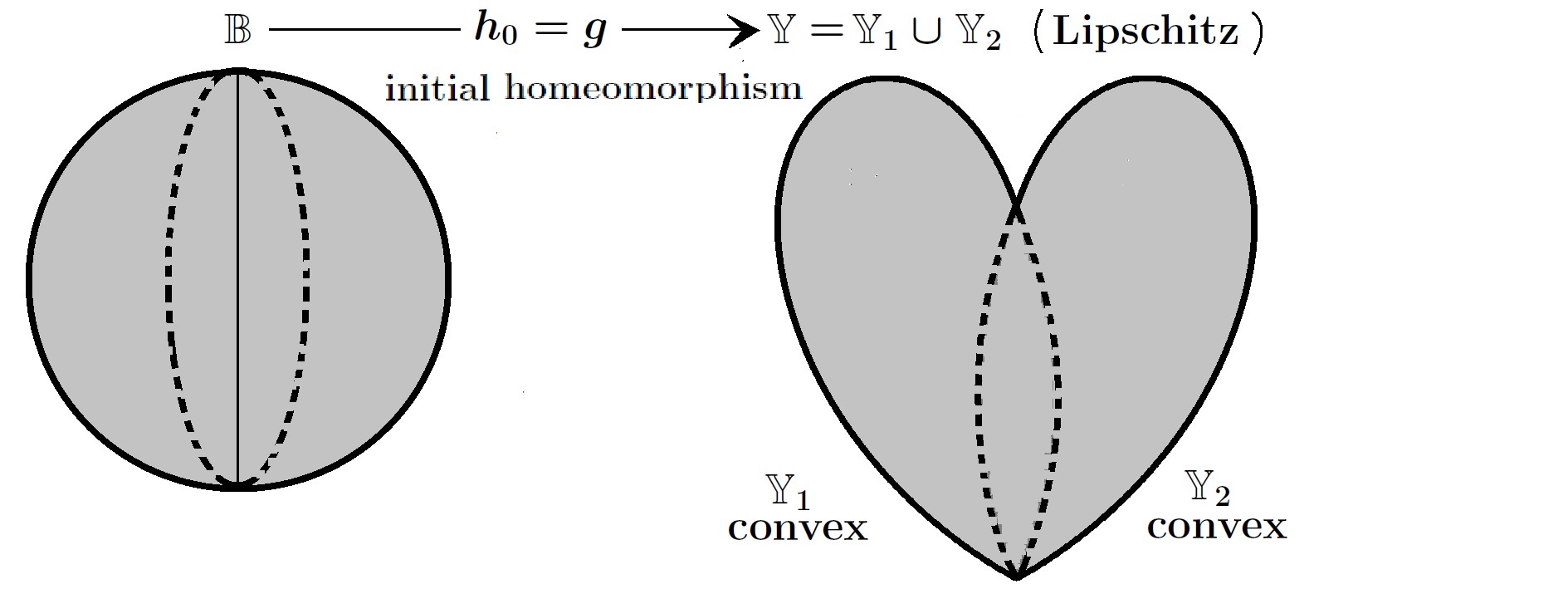}
    \caption{Heart shaped target and the symmetric initial homeomorphism $\, g :  \mathbb B \onto \mathbb Y\,$.}\label{CHeart1}
\end{figure}
Let $\,g \colon \overline{\mathbb B} \onto \overline{\Y}$ be a homeomorphism in the Sobolev class $\W^{1,2} (\mathbb B , \mathbb C)$.
According to Theorem~\ref{thm:main} there is a unique monotone Hopf-harmonic map $\,h \colon  \overline{\mathbb B} \onto \overline{\Y}$ which agrees with $g$ on $\,\partial \mathbb B\,$. To simplify matters further,  we assume that the boundary data $\, g : \partial \mathbb B \onto \partial \mathbb Y\,$ is also symmetric about the vertical axis. Precisely,  $\, g(-x, y) = - \overline{g(x,y)}$. By the arguments similar to those for~\eqref{SymmetryOfH} and~\eqref{eq:53}  It then follows from the uniqueness statement in Theorem~\ref{thm:main} that $\, h(-x, y) = - \overline{h(x,y)}\,$, everywhere in $\,\overline{\mathbb B}\,$. On the other hand, by Theorem \ref{thm:main}  $\,h^{-1}(\mathbb Y)\,$ is a simply connected subdomain of $\,\mathbb B\,$ in which $\, h \,$ is a harmonic diffeomorphism. Such a subdomain must be the entire disk $\,\mathbb B\,$ with a cut (possibly empty) along a segment of the vertical diagonal. Example \ref{ex:butterfly} shows that in general such a cut need not be empty. Figure \ref{CHeart4} illustrates this case (together with the additional features of the limit map of the alternating process).
\begin{figure}[h!]
    \centering
    \includegraphics[width=1.0\textwidth]{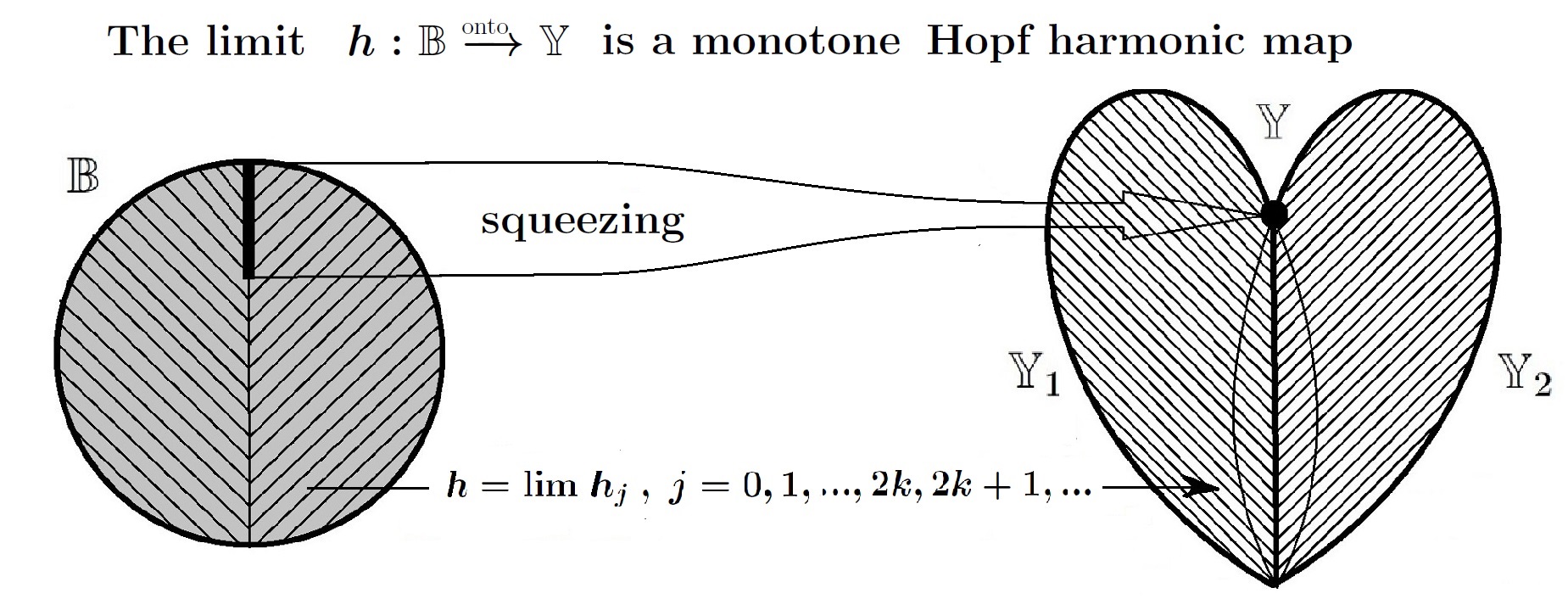}
    \caption{Squeezing phenomenon for a symmetric initial homeomorphism $\, g :  \mathbb B \onto \mathbb Y\,$.}\label{CHeart4}
\end{figure}

The idea below is reminiscent  of the Schwartz alternating process.

\subsection{The iteration process}

We shall construct, by induction, a sequence of homeomorphisms $\,h_j \in \mathscr H_g (\overline{\mathbb B}, \overline{\Y})$. The induction begins with $\,h_0 \equiv g\,$, see Figure \ref{CHeart1}, and continues with mappings denoted by $\,h_1, h_2, \dots , h_{2k-1}, h_{2k}, \dots$ for $k=1,2, \dots\,$.\\

\textsf{Definition of $h_1 \colon \overline{\mathbb B} \onto \overline{\Y}$}
\[h_1 = \begin{cases}
\textnormal{harmonic replacement of } h_0 \colon h_0^{-1} (\Y_1) \onto \Y_1 \\
h_0 \; \textnormal{ in } \overline{\mathbb B} \setminus h_0^{-1} (\Y_1)
\end{cases}\]
Hereafter the term harmonic replacement of a map $\,f \in \mathscr C (\overline{\Omega}, \C)\,$ refers to a map  $\,\tilde{f} \in \mathscr C (\overline{\Omega}, \C)\,$ which is harmonic in $\,\Omega\,$ and coincides with $\,f\,$ on $\,\partial \Omega\,$.

\begin{figure}[h!]
\includegraphics[width=1.0\textwidth]{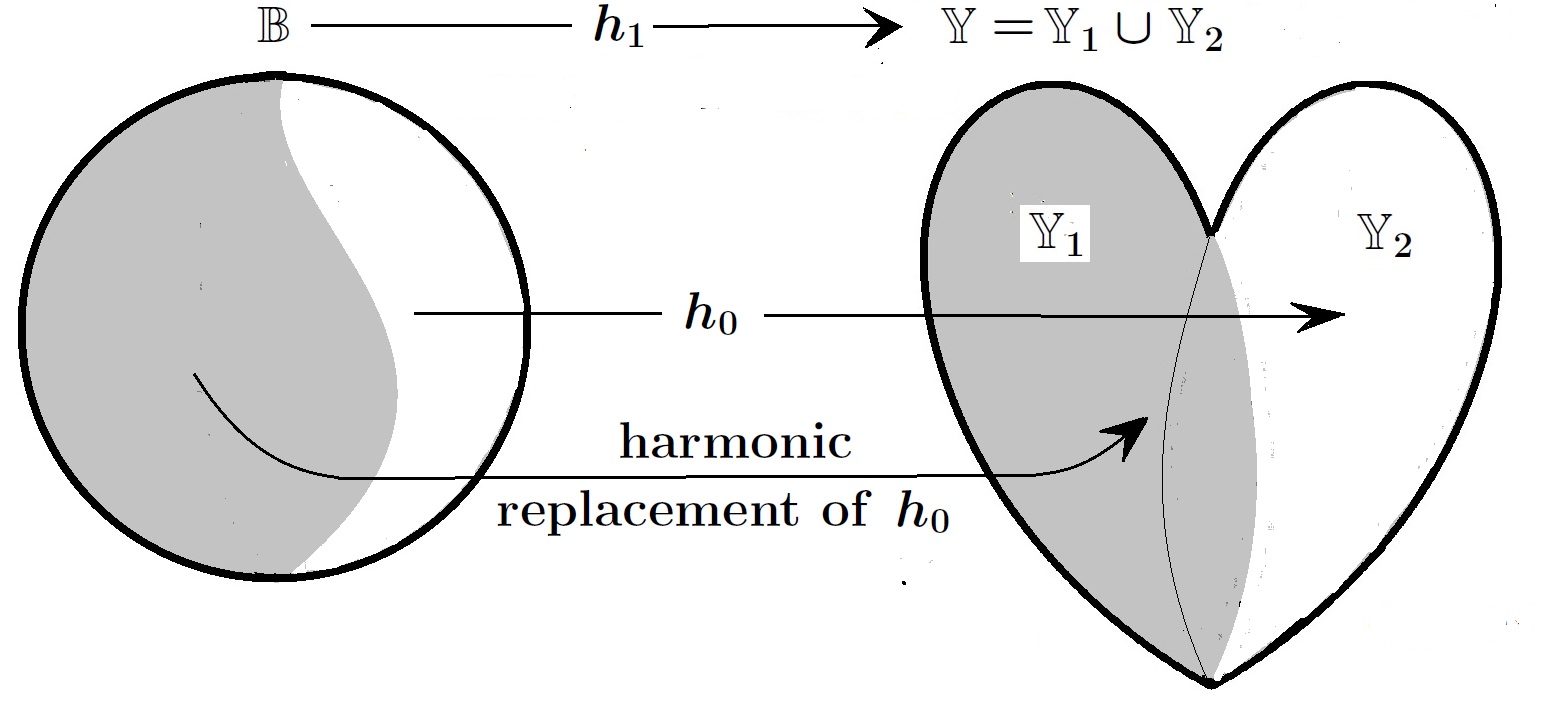}
\caption{First harmonic replacement; the map $h_1 \colon h_0^{-1} (\Y_1) \onto \Y_1$}\label{h1}
\end{figure}
Note that
\[\mathcal E [h_1] = \int_{\mathbb B} \abs{Dh_1}^2 \le  \int_{\mathbb B} \abs{Dh_0}^2 = \mathcal E [h_0]  \, . \]
\textsf{Definition of $h_2 \colon \overline{\mathbb B} \onto \overline{\Y}$}

\[h_2 = \begin{cases}
\textnormal{harmonic replacement of } h_1 \colon h_1^{-1} (\Y_2) \onto \Y_2 \\
h_1 \; \textnormal{ in } \overline{\mathbb B} \setminus h_1^{-1} (\Y_2)
\end{cases}\]

Thus  $\,h_2\,$ is harmonic in  $\,\mathbb B \setminus \partial h_1^{-1} (\Y_2)$.

\begin{figure}[h!]
    \centering
    \includegraphics[width=1.0\textwidth]{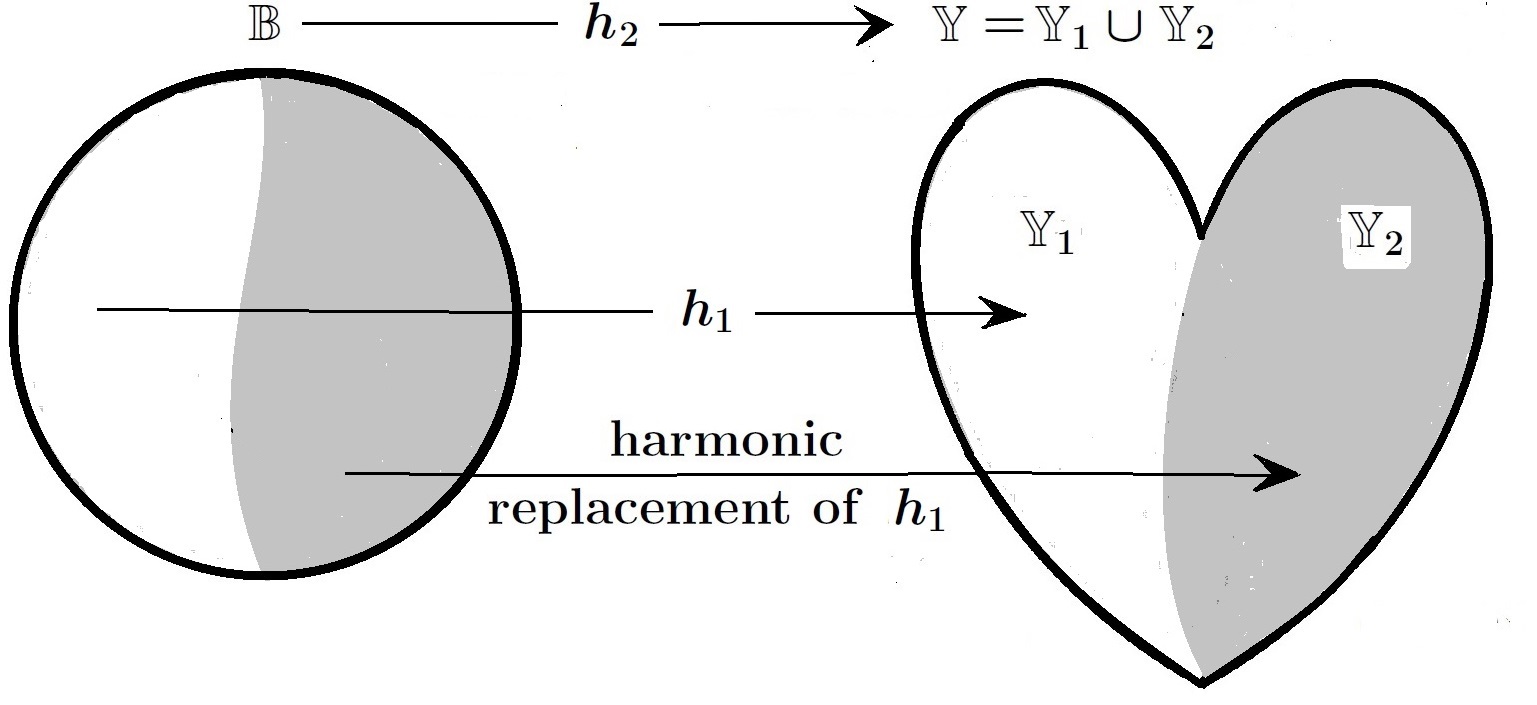}
    \caption{Second harmonic replacement;  the map $\,h_2 \colon h_1^{-1} (\Y_2) \onto \Y_2 $}\label{h2}
\end{figure}

Again, we have
\[\mathcal E [h_2] \le \mathcal E [h_1] \le \mathcal E [h_0] \, .\]

Now, suppose we have defined $h_{2k-1}$ and $h_{2k}$ for some $k \ge 1$.\\
\textsf{Definition of $h_{2k+1} \colon \overline{\mathbb B} \onto \overline{\Y}$,}
\[h_{2k+1} = \begin{cases}
\textnormal{harmonic replacement of } h_{2k} \colon h_{2k}^{-1} (\Y_1) \onto \Y_1 \\
h_{2k} \; \textnormal{ in } \overline{\mathbb B} \setminus h^{-1}_{2k} (\Y_1)
\end{cases}\]
\textsf{Definition of $h_{2k+2} \colon \overline{\mathbb B} \onto \overline{\Y}$,}
\[h_{2k+2} = \begin{cases}
\textnormal{harmonic replacement of } h_{2k+1} \colon h_{2k+1}^{-1} (\Y_2) \onto \Y_2 \\
h_{2k+1} \; \textnormal{ in } \overline{\mathbb B} \setminus h^{-1}_{2k+1} (\Y_2)
\end{cases}\]
In each step of our construction we lower the Dirichlet energy, unless the map $\,h_j\,$ turns out to be harmonic, in which case the process terminates,
\[\mathcal E [h_0] \ge \dots \ge \mathcal E [h_{2k+1}] \ge \mathcal E [h_{2k}] \ge \dots \, , \qquad k=0,1,2, \dots\]

 Furthermore, $h_{2k+1}$ is a harmonic homeomorphism from $h^{-1}_{2k} (\Y_1)$ onto $\Y_1$ and from $\Y \setminus \overline{h^{-1}_{2k} (\Y_1)}$ onto $\Y \setminus  \overline{ \Y_1}$. Similarly,  $h_{2k+2}$ is a harmonic homeomorphism from $h^{-1}_{2k+1} (\Y_2)$ onto $\Y_2$ and from $\Y \setminus \overline{h^{-1}_{2k+1} (\Y_2)}$ onto $\Y \setminus  \overline{ \Y_2} $.
\subsection{The question of convergence}

The family $\{h_j\}$ is equicontinuous. This follows from the uniform bound of the modulus of continuity; namely,
\[ \abs{h_j (x_1) - h_j (x_2)}^2 \le    \frac{C_{\X, \Y}\,\int_{\mathbb B} \abs{Dg(x)}^2  \, \dtext x}{\log \left(e + 1/\abs{x_1-x_2} \right)}  \]
for all $x_1, x_2 \in \overline{\mathbb B}$, see~\eqref{eq:modofcont}. In particular, $\{h_j\}$ contains a subsequence converging uniformly on $\overline{\mathbb B}$. An obvious question to ask is whether the entire sequence  $\{h_j\}$  converges; precisely,

\begin{question}
Does  $\{h_j\}$ converge uniformly (consequently, weakly in $\W^{1,2} (\mathbb B, \mathbb C)$) to a mapping $h \colon \overline{\mathbb B}  \onto \overline{\Y}$ (obviously monotone) of   smallest Dirichlet energy within the class $\mathscr M_g (\overline{\mathbb B} , \overline{\Y})$?
\end{question}
The answer is not known to us in full generality. Whenever the answer to this question is "yes", the limit map $\,h\,$ turns out to be the unique monotone Hopf-harmonic solution, as stated in Theorem~\ref{thm:main}.

\end{document}